\documentclass[12,reqno]{amsart}
\usepackage{fullpage}
\usepackage{amsfonts}
\usepackage[active]{srcltx}

\usepackage{amsmath,amssymb,amsthm,bm}
\usepackage{mathtools}
\usepackage{shuffle}
\usepackage{bm} 
\usepackage[shortlabels]{enumitem}
\usepackage{empheq}
\usepackage{wasysym}
\usepackage{verbatim}
\usepackage{graphicx}
\usepackage[
pagebackref,
bookmarks=true,         
bookmarksnumbered=true, 
colorlinks=true,
pdfstartview=FitV,
linkcolor=blue, citecolor=blue, urlcolor=blue]{hyperref}
\usepackage[backrefs,alphabetic,abbrev,msc-links]{amsrefs}
\usepackage{microtype}
\theoremstyle{plain}
\newtheorem{theorem}{Theorem}[section]

\newtheorem{corollary}[theorem]{Corollary}

\newtheorem{lemma}[theorem]{Lemma}
\newtheorem{problem}[theorem]{Problem}
\newtheorem{proposition}[theorem]{Proposition}

\theoremstyle{remark}

\theoremstyle{definition}
\newtheorem{remark}[theorem]{Remark}
\newtheorem{definition}[theorem]{Definition}
\newtheorem{example}[theorem]{Example}
\newtheorem{hypothesis}[theorem]{Hypothesis}

\def\var{\mathrm{Var}}

\def\RR{\mathbb{R}}
\def\EE{\mathbb{E}}
\def\NN{\mathbb{N}}
\def\WW{\mathbb{W}}

\def\1{{\bf 1}}
\def\cA{{\mathcal A}}

\def\cF{{\mathcal F}}

\def\cH{{\mathcal H}}

\def\de{{\delta}}
\def\la{{\lambda}}
\def\si{{\sigma}}

\def\Om{{\Omega}}

\def\al{{\alpha}}

\def\Ga{{\Gamma}}
\def\ga{{\gamma}}
\def\de{{\delta}}

\def\si{{\sigma}}

\def\la{{\lambda}}

\def\vare{{\varepsilon}}
\def \eref#1{\hbox{(\ref{#1})}}

\def\si{{\sigma}}
\def\al{{\alpha}}

\def\eref#1{\hbox{(\ref{#1})}}

\newcommand{\lc}{\left[}
\newcommand{\rc}{\right]}

\newcommand{\R}{{\mathbb R}}

\newcommand{\HH}{\mathcal H}

\newcommand{\DD}{\mathbb D}

\def\lt{\left}
\def\rt{\right}
\def\its{\ell }

\let\Section=\section
\def\section{\setcounter{equation}{0}\Section}

\title[parabolic Anderson random fields]{Asymptotics of the density of   parabolic Anderson random fields}
\author[Y. Hu]{Yaozhong Hu}
\thanks{Y. Hu is supported by a startup fund from   University of Alberta at Edmonton.}
\address{Department of Mathematical and Statistical Sciences \\
University of Alberta at Edmonton \\
Edmonton, Alberta, T6G 2G1}
\email{yaozhong@ualberta.ca}
\author[K. L\^e]{Khoa  L\^e}
\thanks{K. L\^e is supported by Martin Hairer's Leverhulme Trust leadership award.}
\address{Department of Mathematics \\
Imperial College \\
London, UK, SW7 2AZ}
\email{n.le@imperial.ac.uk}


\begin{document}
\begin{abstract}
We investigate the sharp density $\rho(t,x; y)$   of the    solution $u(t,x)$ to  stochastic partial differential equation 
$\frac{\partial }{\partial t} u(t,x)=\frac12 \Delta u(t,x)+u\diamond  \dot W(t,x)$, where $\dot W$ is a general Gaussian noise and $\diamond$ denotes the Wick product.  We    
  mainly  concern with  the asymptotic behavior of    $\rho(t,x; y)$ when $y\rightarrow \infty$ or when $t\to0+$.  Both   upper   and
lower bounds  are obtained and these two bounds  match each other modulo some multiplicative constants.   If  the initial datum is positive, then $\rho(t,x;y)$ is supported on the positive half line $y\in [0, \infty)$. In this case 
we show that  $\rho(t,x; 0+)=0$   and  obtain an upper bound for $\rho(t,x; y)$ when $y\rightarrow 0+$. 
\end{abstract}
\maketitle 

\section{Introduction}
Let   $(\Om, \cF, P)$  be  a probability space, on which  the expectation is denoted by $\EE$ , and $W=\{ W(t,x)\,, t\ge 0, 
x\in \RR^d \}$ be  a Gaussian random field on $(\Om, \cF, P)$. 
Its  formal derivative
 $\dot{W}=\left\{  \dot W(t,x)=\frac{\partial ^{d+1}}{\partial t\partial x_1\cdots
 \partial x_d}W(t,x) 
\right\}$   is  a  Gaussian noise  field with the following
covariance  structure
\begin{equation}
\EE \left[ \dot W(s,x)\dot W(t,y)\right] =Q(t,s,x,y)=\gamma_0(t-s)\gamma(x-y)\,.
\label{e.1.2}
\end{equation}
The  temporal covariance distribution $\gamma_0$  
is either the Dirac mass at 0 or a locally integrable  function. The spatial covariance distribution $\gamma$ is either the Dirac mass at 0 in $\RR$ or a function whoes has spectral density  $\mu$ on $\RR^d$ satisfying
\[
\gamma(x)=\int_{\RR^d} e^{\imath x\cdot\xi} \mu(\xi)d \xi\,,\quad {\rm with}\quad 
\int_{\RR^d} \frac{\mu(\xi)d \xi}{1+|\xi|^2}<\infty\,.
\]
In particular, the popular case of space-time white noise in $\RR_+\times\RR$ is included.

The main subject we exclusively consider in this paper is the following stochastic heat equation  
\begin{equation}
\begin{cases}\frac{\partial }{\partial t} u_{(\its)}(t,x)=\frac12 \Delta u_{(\its)}(t,x)+ \its u_{(\its)}(t,x) \diamond  \dot W(t,x)\,,&\qquad t\ge 0,x\in\RR^d,\\
u(0,x)=u_0(x)\,,
\end{cases}
\label{e.1.1}
\end{equation}
where $\its$ is a positive constant, $\Delta=\sum_{i=1}^d \frac{\partial^2}{\partial x_i^2}$ is the Laplacian  and 
$\diamond$ denotes the Wick product (see e.g. \cite{HY}). The parameter $\its$ represents  the intensity of the noise and in the case when $\its=1$, we omit it in the notation $u_{(\its)}$, writing $u$ instead.
It is proved recently in \cite{HHNT}  that under the previously described  conditions on the covariance structure of the noise, a random field solution exists uniquely in $L^2(\Omega)$. 
  
For any fixed $t>0$ and $x\in \RR^d$,  $u(t,x)$ is a real valued  random variable.    
 It is natural to  ask   the following questions: 
 \begin{enumerate}[(I)]
 	\item  Is there   a (probability)  density  function  $y\mapsto\rho(t,x; y)$ such that
	\begin{equation}
	P(u(t,x)\in A)=\int_A \rho(t,x; y)dy\,, \quad \forall \ \hbox{Borel set }\ \ A \subseteq \RR\,?
	\end{equation}
	\item If $\rho(t,x;y)$ exists, what is its general shape?
 \end{enumerate}

The present paper give an affirmative answer for (I) under some scaling assumptions on the covariance of the noise. Concerning question (II), we find explicit functions $g_i(t,x;y)$, $i=1,2$ such that $g_1(t,x;y)\le \rho(t,x;y)\le g_2(t,x;y)$ (Theorem \ref{t.7.2}). In addition, modulo some multiplicative constants, the lowest order asymptotic behaviors of $g_1(t,x;y) $ and $g_2(t,x;y) $ match over two asymptotic regimes: a) $(t,x)$ is fixed, $y\to\infty$ and b) $(x,y)$ is fixed such that $y$ is bounded away from 0, $t\to0+$ (see Remark \ref{rem:asymp} below). When the initial condition is nonnegative, the solution $u(t,x)$ is also nonnegative, and hence $\rho(t,x;y)$ is supported on $[0,\infty)$. In this case, another interesting regime appear: $(t,x)$ is fixed and $y\to0+$. We obtain an upper bound for $\rho(t,x;y)$, which differs from previous bounds (Theorem \ref{thm:lefttail}).

Previously, the existence and   smoothness of $\rho(t,x; y)$ have been studied for some specific   noises. 
In \cite{MN},  the smoothness of density is obtained when  the noise is 
space-time white and the spatial dimension is one.  In \cite{HNS}, the smoothness of the density is obtained when
the noise  is  fractional with Hurst parameters $H_0, H_1, 
\cdots, H_d$ satisfying
\begin{equation} 
H_i>1/2\,, \ i=0, 1, \cdots, d\,,\quad {\rm and }\quad 
 2H_0+\sum_{i=1}^d H_i>d+1\,. \label{e.hns}   
\end{equation}
Our assumption on the noise structure is more general. 
We allow the noise structure to be   the one   in \cite{HHNT} which includes the cases considered in \cites{MN,HNS}.

Relevant to question (II), let us mention that there have been already many  works  on the upper and lower bounds
of the density for the solution of stochastic differential equations and stochastic partial differential equations (see e.g. \cites{bally,kohatsu,leru} and references therein).  
However, the bounds for the density 
obtained   in the afore-mentioned  papers are of Gaussian shapes.
It seems that the  approaches  used in all  the  above mentioned papers are not applicable to  our present situation.
In particular, our bounds for the density are not of Gaussian type: they have rather heavy  tail and this is certainly not surprising due to the intermittency property of the solution.  

Let us briefly describe our approach.
To show the existence and smoothness of density by using Malliavin calculus, a key ingredient is to show that the Malliavin covariance matrix has negative moments of all orders. The techniques in previous work (\cites{MN,HNS}) require that the noise in \eqref{e.1.1} is white in time or there is a Feynman-Kac formula for the solution, which are not the case in our consideration. The current paper follows a different route. We show that the Malliavin covariance matrix has finite negative moments by deriving estimates for the small ball probability $P(\|Du(t,x)\|_{\cH}\le a)$ as $a\to0+$ (Theorem \ref{thm:Duneg}). Previously, the (left) tail probability $P(u(t,x)\le a)$ as $a\to0+$ has been studied in \cite{flores} when the noise is space-time white in one dimension using the standard discrete approximation of the Laplacian together with concentration inequalities. In the current paper, we propose another approximation scheme of the Gaussian noise which works in our general setting. Consequently, we obtain estimates on $P(u(t,x)\le a)$ (Theorem \ref{thm:usmallball}). Moreover, our approximation scheme is flexible enough to derive estimates on $P(\|Du(t,x)\|_{\cH}\le a)$ (Theorem \ref{thm:Zsmall}). 
Having established various estimates of positive and negative moments of $u(t,x)$ and its Malliavin derivatives, it is rather straightforward to obtain corresponding bounds on  $\rho(t,x; y)$ (see Section \ref{sec:tails}).

Equation \eqref{e.1.1} has   also been studied when the Wick product $\diamond $ is replaced 
by the usual product (Stratonovich form) under some more restricted 
condition of the noise covariance structure in \cite{HHNT}.  However, to simplify    the presentation, we only
consider  the Wick product case. The Stratonovich case can be treated analogously as long as a square integrable solution exists.

If the Gaussian noise is reduced to the fractional Brownian noise, then 
the assumption that $\gamma_0$ and $\ga$ are positive (generalized) 
functions means that  the Hurst 
parameters must be greater than $1/2$ (long memory case). When the Hurst 
parameter is less than $1/2$ there are some recent work 
on the   equation \eqref{e.1.1} (see e.g. \cite{HHLNT}).  We will not deal with this situation 
in present work since it  is more  involved.  In particular, some estimates herein can not apply directly and modifications are necessary.

Here is the organization of the paper.  In Section \ref{sec:pre}, we  
 briefly  recall some results  from \cite{HHNT}   that we are going to use and also to fix some notations. We refer the readers to that paper for
other   concepts such as stochastic integral, existence and uniqueness 
of  solution etc.   In Section \ref{sec:right}, we state some right tail results which can be obtained by the high moment bounds. For the lower bound, we use the Paley-Zygmund inequality. 
  In Section \ref{sec:left_tail} we extend the  recent result in \cite{flores}
on the left tail asymptotics  in  one dimensional space time white noise  case to
general Gaussian noise in any dimension.        Section \ref{sec:non_degeneracy_of_malliavin_derivatives}  is devoted to the bounds of the negative moments of the Malliavin covariance matrix. 
Section \ref{sec:tails} presents our main results of the paper on the asymptotic behavior of  
the density $\rho(t,x; y)$ as $y\rightarrow \infty$  and as $y\rightarrow 0$.

\section{Preliminary}\label{sec:pre}
%
%

We shall  follow  the assumptions made in \cite{HHNT},   which  we now recall.
The functions $\gamma_0$ and $\gamma$ in \eqref{e.1.2} are general nonnegative and nonnegative
definite  functions or  Dirac delta masses. Since we are interested in qualitative estimates, certain scaling properties of the covariances are assumed below. We suspect that some of our results still holds under the general conditions described in the Introduction, however, we do not pursue this direction.

To be more precise, throughout the paper, the following assumptions are enforced.
\begin{hypothesis} \label{h2}
$\gamma_0=\delta$ is the Dirac delta distribution or $\gamma_0$ is a function and there exist constants $c_0,C_0$ and $0\le \alpha_0<1$ (independent of $t$), such that
\[
c_0  t  ^{-\alpha_0} \le \gamma_0(t) \le C_0  t  ^{-\alpha_0}\,,\quad \forall \ t>0\,. 
\]
\end{hypothesis}

\begin{hypothesis}\label{hspa} $\gamma=\delta$ is the Dirac delta distribution on $\RR$ (this forces $d=1$) or $\gamma$ is a function and there exist constant $\alpha\in(0,2)$ and $c_0>0$ such that $\sup_{c,x:c>0,\gamma(x)\neq0}\frac{\gamma(cx)}{c^{- \alpha} \gamma(x)}<\infty$ and $\inf_{x:|x|\le \varepsilon}\gamma(x)\ge c_0 \varepsilon^{-\alpha}$ for every $ \varepsilon>0$.
\end{hypothesis}
\begin{remark}
	The cases $\gamma_0=\delta$ and $\gamma=\delta$ correspond respectively to white in time noises and white in space noises. If $\gamma_0=\delta$, we set $\alpha_0=1$ and if $\gamma=\delta$, we set $\alpha=1$. The case $\alpha_0=0$ includes noises which are independent of time.
\end{remark}
\begin{example}  If  $\dot W$ is   fractional Gaussian field with Hurst
parameter $H_0$ in time and $H=(H_1, \cdots, H_d)$ in space, then
\[
\gamma_0(t)= R_{H_0}(t)\quad {\rm and} \quad
\gamma(x)=\prod_{i=1}^d  R_{H_i}(x_i)\,,
\]
where $x=(x_1, \cdots, x_d)$ and $R_H(t)=H (2H -1) |t|^{2H -2}$.
 In this case, Hypothesis \ref{h2} is verified with $\alpha_0=2-2H_0$, where $H_0\in (1/2, 1)$.
 Hypothesis \ref{hspa} is verified with $\alpha=2d-2\sum_{i=1}^dH_i $ when
\[
H_i\in (1/2, 1)\,,\  i= 1, \cdots, d\,,\quad {\rm and}\quad  \sum_{i=1}^d H_{i}>d-1\,.
\]
Since $H_0\in(\frac12,1)$, the above condition is implied  by \eqref{e.hns}.  Thus the assumptions made in this paper cover the cases considered in \cite{HNS} and \cite{MN}. 
\end{example}

For the Gaussian noise $\dot W$ satisfying Hypotheses \ref{h2} and  \ref{hspa}, the Cameron-Martin space $\cH$ is defined as the completion of $C^\infty_c(\RR_+\times\RR^d)$ under the scalar product
\begin{equation*}
	\langle \phi,\psi\rangle_{\cH}=\iint_{(\RR_+\times\RR^d)^2} \phi(s,x)\psi(t,y)\gamma_0(t-s)\gamma(x-y)dsdxdtdy \quad\forall \phi,\psi\in C^\infty_c(\RR_+\times\RR^d) \,.
\end{equation*}
The stochastic integral $\int_0^t \int_{\RR^d} f(s, y) W(ds, dy)$ is well-defined for $f$ belongs to $\cH$ and for some random kernel $f$ specified in \cite{HHNT}. We shall freely use this and some other results obtained there. 

In what follows, we denote by $p_t(x)$ the heat kernel $(2\pi t)^{-d/2} e^{-|x|^2/(2t)}$, by
$p_t*f(x)$ the spatial convolution $\int_{\RR^d} p_t(x-y) f(y)dy$, and by $C_b (\mathbb{R}^d)$ the set of all bounded continuous functions from $\RR^d$ to $\RR$.  
\begin{definition}\label{def1}
An adapted   random field $u_{(\its)}=\{u_{(\its)}(t,x) ; t \geq 0, x \in
\mathbb{R}^d\}$ such that $\EE [ u^2_{(\its)}(t,x)] < \infty$ for all $(t,x)$ is  called 
a mild solution to equation \eref{e.1.1} with initial condition $u_0 \in C_b (\mathbb{R}^d)$, if for any $(t,x) \in [0,
\infty)\times \mathbb{R}^d$, the process $\{p_{t-s}(x-y)u(s,y) {\bf
1}_{[0,t]}(s); s \geq 0, y \in \mathbb{R}^d\}$ is Skorohod
integrable,  and the following equation holds
\begin{equation}\label{eq:sko-mild}
u_{(\its)}(t,x) =p_t *u_0(x)+\its\int_0^t\int_{\mathbb{R}^d}p_{t-s}(x-y)u_{(\its)}(s,y)  \, 
W(ds, dy)\,. 
\end{equation}
\end{definition}

%
%

\begin{theorem}\label{t.hhnt}
  Assume that initial condition $u_0$ is a measurable function such that $p_t*|u_0|(x)<\infty$ for all $t>0$ and $x\in\RR^d$.
  Suppose that $\gamma_0,\gamma$ satisfy Hypotheses \ref{h2} and \ref{hspa}.
  Then  we have the following statements.
\begin{enumerate}
 \item[(i)]\ The equation \eref{e.1.1}  has a unique   mild solution  which admits the following chaos expansion:
  \begin{equation}\label{eqn:chaos}
  u_{(\its)}(t,x)=\sum_{n=0}^\infty \its^n I_n(f_n(t,x))\,,
  \end{equation}
  where
  \begin{align}\label{def:fn}
  f_n(t,x; s_1,x_1,\dots,s_n,x_n )
   =\frac{1}{n!}p_{t-s_{\si(n)}}(x-x_{\si(n)})\cdots p_{s_{\si(2)}-s_{\si(1)}}(x_{\si(2)}-x_{\si(1)})
  p_{s_{\si(1)}}*u_0(x_{\si(1)})\,,
  \end{align}
  and $I_n(f_n(t,x))$ is the    multiple  Wiener-It\^o  integral with respect to the kernel $f_n(t,x, \cdot)$.
  Here $\si$ denotes the   permutation 
  of $\{1,2,\dots,n\}$ such that $0<s_{\si(1)}<\cdots<s_{\si(n)}<t$. In addition, for every $n\ge1$ and $(t,x)\in\RR_+\times\RR^d$, we have
  \begin{equation}
    \EE\lt[ |I_n(f_n(t,x))|^2\rt]\le C  |p_t*|u_0|(x)|^2  t^{(2- \alpha_0-\frac \alpha2)n}(n!)^{\frac \alpha2-1}\,.
    \label{e.2.4}
  \end{equation}
  
 \item[(ii)]\ The solution $u_{(\its)}(t,x)$ has all moments and there are 
  positive constants  $C_1$ and $C_2$ such that
  \begin{equation}   \label{est:pmup}
  \EE \lc \left|u_{(\its)}(t,x)\right| ^p\rc\leq  [C_1 p_t*|u_0|(x) ]^p  \exp\left(C_2 \its^{\frac4{2- \alpha}}
  t^{\frac{4-2\alpha_0-\alpha}{2-\alpha}}p^{\frac{4-\alpha}{2-\alpha}}\right)
  \end{equation}
  for all $t\ge 0\,, x\in \R^d\,, p\ge 1$ and $\its>0$.
  
\item[(iii)]\ If the initial condition $u_0(x)$ is bounded from below by a positive constant, namely, there
  is a constant $L>0$ such that $u_0(x)\ge L$,  then the solution $u_{(\its)}(t,x)$ is also positive almost surely and  there are positive constants $\tilde C_1$ and $\tilde C_2$,
  independent of $t$, $x$,  and $p$,    such that  
  \begin{equation}   \label{est:pmlow}
  \EE \lc u_{(\its)}(t,x) ^p\rc\geq  \tilde C_1^p\exp\left(\tilde C_2 \its^{\frac4{2- \alpha}}
  t^{\frac{4-2\alpha_0-\alpha}{2-\alpha}}p^{\frac{4-\alpha}{2-\alpha}}\right)
  \end{equation}
  for all $t\ge 0\,, x\in \R^d\,, p>1$.
\end{enumerate} 
\end{theorem} 
\begin{remark} By an approximate procedure, Theorem \ref{t.hhnt}(i) and (ii) hold for initial conditions which are Dirac masses.
\end{remark}
\begin{proof}
	(i) Existence and uniqueness of a random field solution have been obtained in \cite{HLN2}. The chaos expansion \eqref{eqn:chaos} is also obtained there. The estimate \eqref{e.2.4} is implicit in \cites{HLN,HLN2} and \cite{HHNT}. We give a brief argument below and make references to the afore-mentioned papers whenever necessary.

 	We begin with an estimate extracted from \cite{HLN2}*{inequality (3.9)}
	\begin{align}\label{tmp:hln}
		n!\|f_n(t,x)\|^2_{\cH^{\otimes n}}\le C^n|p_t*u_0(x)|^2\int_{[0,t]^n_<}\int_{\RR^{nd}}\exp\lt\{-\var \lt(\sum_{j=1}^n \xi_j\cdot B_{0,t}^j(s_j) \rt) \rt\}\mu(\xi)d \xi d s
	\end{align}
	where $C$ is some positive constant, $[0,t]^n_<=\{(s_1,\dots,s_n)\in[0,t]^n:0<s_1<\cdots<s_n<t \} $, $B^j_{0,t}=\{B^j_{0,t}(s),s\in[0,t]\} $, $j=1,\dots,n$ are independent Brownian bridges which start and end at 0, $\mu(\xi) d \xi=\prod_{j=1}^n \mu(\xi_j) d \xi_j$ and $ds=\prod_{j=1}^n ds_j $. The constant $C$ is not explicitly stated in \cite{HLN2}, however, it can be easily computed from \cite{HLN2}*{page 623}, 
	\begin{equation*}
		C=\int_0^t \gamma_0(s)ds\,.
	\end{equation*}
	From Hypothesis \ref{h2}, we see that $C\le C_0(1- \alpha_0)^{-1} t^{1- \alpha_0}$.
	Using scaling of Brownian bridges and Hypothesis \ref{hspa}, we obtain from \eqref{tmp:hln} that
	\begin{equation*}
		n!\|f_n(t,x)\|^2_{\cH^{\otimes n}}\le \tilde C^n t^{(2- \alpha_0-\frac \alpha2)n} |p_t*u_0(x)|^2\int_{[0,1]^n_<}\int_{\RR^{nd}}\exp\lt\{-\var \lt(\sum_{j=1}^n \xi_j\cdot B_{0,1}^j(s_j) \rt) \rt\}\mu(\xi)d \xi d s
	\end{equation*}
	for some constant $\tilde C$. Since $\EE[I_n(f_n(t,x))]=n!\|f_n(t,x)\|^2_{\cH^{\otimes n}}$, in order to obtain \eqref{e.2.4}, it remains to show that
	\begin{equation}
		K_n:=\int_{[0,1]^n_<}\int_{\RR^{nd}}\exp\lt\{-\var \lt(\sum_{j=1}^n \xi_j\cdot B_{0,1}^j(s_j) \rt) \rt\}\mu(\xi)d \xi d s
		\le \bar C^n (n!)^{\frac \alpha2-1}
	\end{equation}
	for some constant $\bar C$. In what follows, the symbol $\bar C$ denotes some positive constant which can vary from line to line. From \cite{HLN}*{eq. (9.85)} we have
	\begin{equation*}
		K_n=\frac1{n!}\EE\lt[\int_0^1 \gamma(\sqrt2 B_{0,1}(s))ds \rt]^n\,.
	\end{equation*}
	By the elementary inequality $(a+b)^n\le 2^{n-1}( a^n+b^n)$ and the fact that $B_{0,1}\stackrel{\mathrm{law}}{=}B_{0,1}(1-\cdot)$, we have
	\begin{equation*}
		K_n\le \frac{2^{n-1}}{n!} \EE\lt[\int_0^{\frac12} \gamma(\sqrt2 B_{0,1}(s))ds \rt]^n\,.
	\end{equation*}
	Using \cite{HLN}*{eq. (2.38)} (or Girsanov theorem), 
	\begin{equation*}
		\EE\lt[\int_0^{\frac12} \gamma(\sqrt2 B_{0,1}(s))ds \rt]^n\le 2^{\frac d2}\EE\lt[\int_0^{\frac12} \gamma(\sqrt2 B(s))ds \rt]^n\,,
	\end{equation*}
	where $B$ is a Brownian motion. We now write
	\begin{align*}
		\frac1{n!}\EE\lt[\int_0^{\frac12} \gamma(\sqrt2 B(s))ds \rt]^n=\int_{[0,\frac12]^n_<}\int_{\RR^{nd}}\prod_{j=1}^n e^{-(s_{j+1}-s_j)|\xi_j+\cdots+\xi_1|^2 }\mu(\xi) d \xi ds
	\end{align*}
	with the convention $s_{n+1}=\frac12$ and use Hypothesis \ref{hspa} to see that
	\begin{equation*}
		K_n\le \bar C^n \int_{[0,\frac12]^n_<}\prod_{j=1}^n(s_{j+1}-s_j)^{-\frac \alpha2}ds\int_{\RR^{Nd}}\prod_{j=1}^n e^{-|\xi_j+\cdots+\xi_1|^2}\mu(\xi)d \xi\,.
	\end{equation*}
	The estimate \eqref{tmp:hln} follows from the above estimate after observing the following inequalities, 
	\begin{align*}
		\int_{[0,\frac12]^n_<}\prod_{j=1}^n(s_{j+1}-s_j)^{-\frac \alpha2}ds\le \frac{\bar C^n}{\Gamma(n(1-\frac \alpha2))}
	\end{align*}
	and
	\begin{align*}
		\int_{\RR^{Nd}}\prod_{j=1}^n e^{-|\xi_j+\cdots+\xi_1|^2}\mu(\xi)d \xi=\lt(\int_{\RR^d}e^{-|\xi_1|^2}\mu(\xi_1)d \xi_1\rt)^n\,,
	\end{align*}
	where $\Gamma$ is the Gamma function.

	Part (ii) is a consequence of part (i).   In fact, from 
	the  hypercontractivity inequality  (\cite{hubook}) and then from
	\eqref{e.2.4}, we have
	\[
	\|I_n(f_n(t,x))\|_p\le (p-1)^{n/2} \|I_n(f_n)\|_2
	\le  C  p_t*|u_0|(x)  p^{\frac n2}  t^{\frac{2- \alpha_0-\frac \alpha2}{2} n}(n!)^{\frac \alpha4-\frac12}\,. 
	\]
	Thus by the asymptotic property of the Mittag-Leffler function, we have 
	\begin{align*}
	\| u_{(\its)}(t,x)\|_p 
	&\le 
	\sum_{n=0}^\infty  \its ^n \|I_n(f_n(t,x) )\|_p\\
	&\le  C  p_t*|u_0|(x) \sum_{n=0}^\infty   \its^n p^{\frac n2}  t^{\frac{2- \alpha_0-\frac \alpha2}{2} n}(n!)^{\frac \alpha4-\frac12}\\
	&\le  C  p_t*|u_0|(x)  \exp\left(C \its^{\frac4{2- \alpha}}
	  t^{\frac{4-2\alpha_0-\alpha}{2-\alpha}}p^{\frac{2 }{2-\alpha}}\right)
	  \end{align*}
	which is equivalent to  \eqref{est:pmup}.   

	(iii) Nonnegativity of $u(t,x)$ has been observed in \cite{HLN2}*{Remark 4.5}. The estimate \eqref{est:pmlow} can be obtained as in \cite{HHNT}*{Theorem 6.4 and Remark 6.6}.
\end{proof}
    
\textbf{Convention:} throughout the paper, we denote $\beta=\frac{4-2 \alpha_0- \alpha}{2- \alpha}$.

\begin{corollary}\label{cor.Dku-bound} Let the assumptions in Theorem \ref{t.hhnt}  be satisfied. 
  For every $k\in\NN$, the solution of \eqref{e.1.1} is $k$-times Malliavin differentiable. In addition, the $k$-th Malliavin derivative of $u_{(\its)}$, denoted by $D^k u_{(\its)}$ has finite moment of all orders and satisfies
  \begin{equation}
    \EE [\|D^k u_{(\its)}(t,x)\|_{\cH^{\otimes k}}^p]\le C_{k,p} |p_t*|u_0|(x)|^p t^{\frac{4- 2\alpha_0-\alpha}4pk} \exp\lt\{c\its^{\frac4{2- \alpha}}p^{\frac{4- \alpha}{2- \alpha}}t^\beta \rt\} 
    \quad\textrm{for all}\quad (t,x)\in\RR_+\times\RR^d
  \end{equation}
  for some  positive constants $c,C_{k,p}$ independent of $\its$.
\end{corollary}
\begin{proof}
  From \cite{nuabook}*{page 28}, $D^ku_{(\its)}(t,x) $ exists and has finite second moment if and only if
  \begin{equation*}
    \sum_{n=1}^\infty n^k \its^{2n}\EE |I_n(f_n(t,x))|^2<\infty\,.
  \end{equation*}
  Using the elementary inequality $n^k\le c_k 4^n$, we see that the series above is at most
  \begin{equation*}
    c_k \sum_{n=1}^\infty (2\its)^{2n}\EE |I_n(f_n(t,x))|^2=c_k\EE [|u_{2\its}(t,x)|^2]\,.
  \end{equation*}
  which is finite by Theorem \ref{t.hhnt} (ii). Next, we obtain estimates for higher moments. Setting
  \begin{equation*}
     J_ku_{(\its)} =\sum_{n=k}^{\infty}\its^nI_n(f_n)
     \quad\textrm{and}\quad Iu_{(\its)}=u_{(\its)}
  \end{equation*} 
  and applying Meyer's inequality, we have
  \begin{equation*}
  \|\|D^k u_{(\its)}(t,x)\|_{\cH^{\otimes k}}\|_p
  =\|\|D^kJ_k u_{(\its)}(t,x)\|_{\cH^{\otimes k}}\|_p
  \le c_{k,p} \|(I-L)^{k/2} J_ku_{(\its)}(t,x)\|_p\,,
  \end{equation*} 
  where $L=-\de D$ is the Ornstein-Uhlenbeck operator. In particular, we have
  \[
  -LJ_ku_{(\its)}(t,x)  =\sum_{n=k}^\infty n  \its^n I_n(f_n(t,x))\,. 
  \]
  Let $q=\sqrt{\frac{1}{p-1}}$ and $\Ga(q)$ be the second quantization operator, namely,
  \[
  \Ga(q)J_k u_{(\its)}(t,x)=\sum_{n=k}^\infty q^n \its^n   I_n(f_n(t,x)) \,.
  \]
  Then by the hypercontractivity (see \cite{hubook}, \cite{nuabook}) and the elementary inequality $(1+n)^k\le k! e^{1+n}$, we have 
  \begin{eqnarray*} 
  && \|(I-L)^{k/2}J_k u_{(\its)}(t,x)\|_p^2
  = \|\Ga(q) (I-L)^{k/2}J_k u_{\its/q} (t,x)\|_p ^2 \\
   & & \qquad \le \|  (I-L)^{k/2}J_k u_{\its/q} (t,x)\|_2 ^2= \sum_{n=k}^\infty (1+n)^{k }\frac{\its^{2n}  }{q^{2n}}   \|
   I_n(f_n(t,x)) \|_2^2 \\
   &&\qquad \le  \sum_{n=k}^\infty k!e^{1+n} \frac{  \its^{2n} }{q^{2n}}   \|
   I_n(f_n(t,x)) \|_2^2  \,.
  \end{eqnarray*} 
  Applying Theorem \ref{t.hhnt} (i) and the elementary inequality $n!\ge k!(n-k)!$ yields
  \begin{align*}
    \|(I-L)^{k/2}J_k u_{(\its)}(t,x)\|_p^2
    &\le c_k|p_t*|u_0|(x)|^2 \sum_{n=k}^\infty \frac{(e\its^2)^n}{q^{2n}}  t^{(2- \alpha_0-\frac \alpha2)n}(n!)^{\frac \alpha2-1}
    \\&\le c_{k,p}|p_t*|u_0|(x)|^2\its^{2k}t^{(2- \alpha_0-\frac \alpha2)k}  \sum_{n=0}^\infty \frac{(e\its^2)^n}{q^{2n}}  t^{(2- \alpha_0-\frac \alpha2)n}(n!)^{\frac \alpha2-1}
    \\&\le c_{k,p}|p_t*|u_0|(x)|^2\its^{2k}t^{(2- \alpha_0-\frac \alpha2)k} \exp\lt\{c \its^{\frac4{2- \alpha}}q^{-\frac4{2- \alpha}} t^\beta\rt\}\,.
  \end{align*}
  This implies the result.
\end{proof}

\section{ Right  tail  probability }\label{sec:right}

For a positive random variable $X\ge 0$, the right tail probability
$P(X\ge K)$ for sufficiently large $K$ can be obtained from upper bounds
of the moments with the help of Chebyshev inequality.
To obtain the lower bound,  
one needs to use    the Paley-Wiener inequality 
which we recall as the following lemma. 


\begin{lemma}[Paley-Zygmund inequality]
If $X \ge 0$  is a random variable with finite variance,  then  for any $0<\theta<1$,  we have 
\begin{equation}
    \operatorname{P}( X \ge \theta\operatorname{E}[X] ) \ge (1-\theta)^2 \frac{\left(\operatorname{E}[X]\right) ^2}{\operatorname{E}[X^2]}\,.\label{e.pz}
\end{equation}
\end{lemma}  	
\begin{proof} This inequality is a well-known.  Since its proof is short we reproduce it here for the convenience of readers.  
First, we have
\[    \operatorname{E}[X] = \operatorname{E}[ X \, \mathbf{1}_{\{ X < \theta \operatorname{E}[X] \}}] + \operatorname{E}[ X \, \mathbf{1}_{\{ X \ge \theta \operatorname{E}[X] \}} ]\,.
\]
The first addend is at most $\theta \operatorname{E}[X]$.  
From the   Cauchy-Schwarz inequality it follows that the second addend is bounded by  $\left[\operatorname{E}(X^2) \operatorname{P}( X\ge \theta\operatorname{E}[X])\right]^{1/2}$.    
This yields  $(1-\theta)\operatorname{E}[X] \le \left[\operatorname{E}(X^2) \operatorname{P}( X \ge \theta\operatorname{E}[X])\right]^{1/2}$, which is equivalent to \eqref{e.pz}.
\end{proof}

%
%

\begin{proposition}\label{p.3.5}  Let $\rho>1$ be a given  
number. If a positive random variable $X$ satisfies
\begin{equation}
\tilde \kappa_1^p 
  \exp\left( \tilde \kappa _2 p^\rho\right)  \le \EE \left[ X^p\right] \le  \kappa _1^p \exp\left( \kappa _2 p^\rho\right)\quad\mbox{for all }p>1\,,
\label{e.lower-cond}
\end{equation}
where   $0<\tilde \kappa _1 \le \kappa _1<\infty$ and $0<\tilde  \kappa _2\le
 \kappa _2<\infty$ are   positive constants independent of $p$. Then
 \begin{equation}
P(X  \ge a) \le  \exp\left\{- \rho^{\frac{\rho}{1-\rho}}\left(\rho-1\right) \kappa _2^{\frac1{1-\rho}}\left(\log \frac a{\kappa_1}\right)^{\frac{\rho}{\rho-1}}\right\}\quad\quad \hbox{for all $a> \kappa_1 e^{\rho \kappa_2} $  }\,, \label{e.upper-tail-general-1}
\end{equation}
and
\begin{equation}
P(X  \ge a) \ge \frac1{4}\exp\lt\{-\lt(2\log\frac{\kappa_1}{\tilde \kappa_1}+2^\rho \kappa_2-2\tilde \kappa_2\rt)\lt(\frac1{\tilde \kappa_2}\log\frac{2a}{\tilde \kappa_1} \rt)^{\frac \rho{\rho-1}} \rt\} \quad \hbox{for all $a>\frac12 \tilde \kappa_1 e^{\tilde \kappa_2}$  }\,. \label{e.upper-tail-general-2}
\end{equation}
\end{proposition}
\begin{proof} 
  Before discussing the detail, we note that the factor $2\log\frac{\kappa_1}{\tilde \kappa_1}+2^\rho \kappa_2-2\tilde \kappa_2$ is positive.

  \textbf{Upper bound:} For any $a>0$ and any $p> 1$, we use the Chebyshev inequality to obtain
  \[
  P(X\ge a)\le \EE(X^p)a^{-p}\le  \exp\left( \kappa _2 p^\rho -p\log \frac a{\kappa_1}\right)\,.
  \]
  If $a> \kappa_1 e^{\rho \kappa_2}$, we can choose $p=\left(\frac{\log a}{\rho \kappa_2}\right)^{1/(\rho-1)}$, which minimizes the right-hand side in the above inequality. This yields \eqref{e.upper-tail-general-1}.

  \textbf{Lower bound:} 
  Let $a>\frac12 \tilde \kappa_1 e^{\tilde \kappa_2}$ be fixed and  let $p>1$ be such that $a=\frac12 \tilde \kappa_1 e^{\tilde \kappa_2 p^{\rho-1}} $. From the Paley-Zygmund inequality \eqref{e.pz}, it follows that
  \begin{align*}
    P(X\ge 2^{-1}\|X\|_p)&=P(X^p\ge 2^{-p}\EE (X^p))
    \ge(1-2^{-p})^2 \frac{[\EE(X^p)]^2}{\EE(X^{2p})}
    \\&\ge \frac{\tilde \kappa_1^{2p}}{4 \kappa_1^{2p}}e^{-(2^\rho \kappa_2-2\tilde \kappa_2) p^{\rho}}
    \ge \frac14 e^{-(2\log\frac{\kappa_1}{\tilde \kappa_1}+2^\rho \kappa_2-2\tilde \kappa_2)p^\rho}\,.
  \end{align*}
  We observe the following facts: $2\log\frac{\kappa_1}{\tilde \kappa_1}+2^\rho \kappa_2-2\tilde \kappa_2$ is positive, $2^{-1}\|X\|_p\ge 2^{-1}\tilde \kappa_1 e^{\tilde \kappa_2 p^{\rho-1}}=a$
  , $P(X\ge a)\ge P(X\ge 2^{-1}\|X\|_p)$ and $p=\lt(\frac1{\tilde \kappa_2}\log\frac{2a}{\tilde \kappa_1} \rt)^{\frac1{\rho-1}}$. Together with the above estimate, these observations imply \eqref{e.upper-tail-general-2}.
\end{proof}

Now Theorem \ref{t.hhnt} can be combined with Proposition \ref{p.3.5} to yields
\begin{theorem}\label{t.3.12} Let the initial condition $u_0(x)$ be bounded from above
and from below by two positive constants.
Suppose that  $\gamma_0,\gamma$ satisfy Hypotheses \ref{h2} and \ref{hspa}.
Then, there are positive constants $a_0,b_0$, $c_j$, $\tilde c_j$, $j=1,2,3$ (independent of
$t$ and $a$)  such that
\begin{align}
\tilde c_1 \exp\left(-\tilde c_2 t^{-\frac{4-2 \alpha_0- \alpha}2 }(\log (\tilde c_3a))^{\frac{4-\alpha} {2}}\right)\le P(u(t,x)\ge a) \le c_1 \exp\left(-c_2 t^{-\frac{4-2 \alpha_0- \alpha}2 }(\log (c_3 a))^{\frac{4-\alpha} {2}}\right)
\label{e.lower-upper-tail}
\end{align}
for all   $a\ge a_0 e^{b_0t^\beta}$.  
\end{theorem}

\begin{remark} (i)\ Again this theorem holds true when $\alpha_0=1$.
(ii)\ From $\alpha_0<1$ and $\alpha<2$ it follows  that the exponent $-\frac{4-2 \alpha_0- \alpha}2$ of $t$ is negative. 
When $t\rightarrow 0$, both the upper and lower bounds in 
\eqref{e.lower-upper-tail} go to zero. In fact, when $t=0$, $u(0,x)$ is bounded above uniformly 
by a 
constant. So, when $a$ is sufficiently large, $P(u(0,x)\ge a)=0$.   This coincides with the limit  determined by both left and right sides of \eqref{e.lower-upper-tail}
as $t\rightarrow 0$. 
It is worth pointing out that  for  any fixed   $a\ge a_0$  the probability 
 $P(u(t,x)\ge a)$  goes to zero  as $t\downarrow 0$ 
exponentially fast
and the speed in the logarithm scale is precisely given by $\frac{4-2 \alpha_0- \alpha}2$. 
\end{remark}

\section{Left  tail  probability }\label{sec:left_tail}
In the previous section, we have obtained the (right) tail bound $P(u(t,x)\ge a)$ when $a\rightarrow \infty$. If 
the initial condition $u_0(x)\ge 0$,   then  it is known that the solution $u(t,x)$  to 
\eqref{e.1.1}  is also nonnegative for all $t\ge 0$ and $x\in \RR^d$.  This means 
that $P(u(t,x)\le 0)=0$.
In this case, it is  
interesting to know the left tail asymptotics of $P(u(t,x)\le a)$ when $a\downarrow 0$.  In addition,
if the law of $u(t,x)$ has a  density   $\rho(t,x; y)$,  then  the support of $\rho(t,x;\cdot)$
is contained in $[0, \infty)$.  Namely, when $y<0$, $\rho(t,x; y)=0$.
It is interesting to know if $\lim_{y\rightarrow 0} \rho(t,x;y)=0$ 
or not and if yes, what is the asymptotics of $\rho(t,x; y)$ when $y\rightarrow 0$.  
In this section, we study the left tail probability 
$P(u(t,x)\le a)$ as $a\downarrow 0$ and in Section \ref{sec:tails}, we shall study asymptotic behavior of $\rho(t,x;y)$ as $y\downarrow 0$. 

When the noise  $\dot W$ is space time white and when $d=1$, there are several methods that have been used for the left tail probability and 
 for the negative moment of Malliavin covariance matrix. One is the method given by Mueller and Nualart
\cite{MN} (for the existence of 
negative moment of Malliavin covariance matrix). One recent result on the left tail is  given in \cite{flores} with the use of random walk approximation  of the white noise field. This method produces  the best-known upper bound for the left tail probability of the solution.  
 We shall work with general  dimension and with general Gaussian noise covariance structure. Our method is inspired by the one in \cite{flores} but with a different approximation. 
 
We begin by describing our Gaussian noise and its approximation.  
Fix a $T>0$.  Let $\WW=\{\WW(t,x)\,,t\in\RR\,, x\in \RR^d\}$ be a space time Brownian field.  This means that $\lt\{\dot \WW(t,x)=
\frac{\partial ^{d+1}}{\partial t\partial x_1\cdots\partial x_d} 
\WW(t,x) \,, t\in\RR\,, x\in \RR^d\rt\}$
is    the space-time white noise.  We consider the following addition assumption on the noise structure which is satisfied in most cases.
\begin{hypothesis}\label{hroot}
  There exist $\eta_0,\eta$ which are either Dirac masses at 0 or measurable functions (on $\RR,\RR^d$ respectively) such that $\gamma_0=\eta_0*\eta_0$ and $\gamma=\eta*\eta$.
\end{hypothesis}

We assume that the Gaussian noise appeared in \eqref{e.1.1} is given by
\begin{equation}
  \label{e.4.2-noise-representation}
  \dot{W} (t,x)=\int_\RR\int_{\RR^d} \eta_0(t-r)\eta(x-z)\WW(dr,dz)\,.
\end{equation}
It is straightforward to verify that $\dot{W}$ given by the above formula indeed satisfies \eqref{e.1.2}.
Now we consider the following approximation of the space time white noise 
\begin{equation}
\dot \WW_{\vare, \de} (s,y) = e^{-\frac{\vare |y|^2}{2}-\frac{\delta|s|^2}2}  \int_\RR\int_{\RR^d}  p_\de(s-r) p_\vare( y-z) 
 \WW(dr,dz)\,. 
 \label{e.appr-ww}
\end{equation}
This approximation induces an approximation of the driving Gaussian noise $\dot W$  in the equation \eqref{e.1.1} in a natural way:
\begin{align}
\dot W_{\vare, \de} (t,x) 
=\int_\RR \int_{\RR^d} \eta_0(t-s)\eta(x-y) \dot \WW_{\vare, \de} (s,y) dsdy
= \int_\RR \int_{\RR^d} \eta_{0,\delta}(t,r) \eta_\varepsilon(x,z) \WW(dr, dz)\,,
\label{e.appr-w}
\end{align}
where
\begin{equation}\label{e.appr-q}
  \eta_{0,\delta}(t,r)=\int_\RR \eta_0(t-s)p_{\delta}(s-r)e^{-\frac{\delta|s|^2}2}ds
  \quad\textrm{and}\quad \eta_{\varepsilon}(x,z) =\int_{\RR^d}\eta(x-y)p_{\varepsilon}(y-z)e^{-\frac{\varepsilon|y|^2}2}dy\,.
\end{equation}
We note that $(t,x)\mapsto\dot{\WW}_{\varepsilon,\delta}(t,x)$ and $\dot{W}_{\varepsilon,\delta}(t,x)$ are well-defined functions.
The covariance of $\dot W_{\vare, \de} $ is 
\begin{equation}\label{def:Qepdel}
Q_{\vare, \de}(t_1, t_2, x_1, x_2)
:= \EE \left[ \dot W_{\vare, \de}(t_1, x_1)   \dot W_{\vare, \de}(t_2, x_2) 
\right]
= \int_\RR \eta_{0,\delta}(t_1,r)\eta_{0, \delta}(t_2,r)dr\int_{\RR^d}\eta_\varepsilon(x_1,z)\eta_{\varepsilon}(x_2,z) dz\,.
\end{equation}
An important and relevant property of this approximation is 
 that $(t,x)\mapsto \dot{\WW}_{\varepsilon,\delta}(t,x)$ belongs to $L^2(\RR\times\RR^d)$ almost surely. In fact, from \eqref{e.appr-ww}, we have
\begin{align*}
  \EE\int_\RR\int_{\RR^d}|\dot{\WW}_{\varepsilon,\delta}(s,y)|^2dsdy=\int_\RR\int_{\RR^d}e^{-\varepsilon|y|^2- \delta|s|^2}p_{2 \delta}(0)p_{2 \varepsilon}(0)dsdy\,,
\end{align*}
which is finite.

By bounding $e^{-\frac{\delta|s|^2}2}$ by 1 in \eqref{e.appr-q}, we see that
\begin{align*}
  \int_\RR \eta_{0,\delta}(t_1,r)\eta_{0, \delta}(t_2,r)dr
  \le\int_\RR [\eta_0*p_\delta(t_1-r)] [\eta_0*p_{\delta}(t_2-r)]dr
  =\eta_0*\eta_0*p_{2 \delta}(t_1-t_2)=\gamma_0*p_{2 \delta}(t_1-t_2) \,.
\end{align*}
It follows that
\begin{equation*}
   \int_\RR \eta_{0,\delta}(t_1,r)\eta_{0, \delta}(t_2,r)dr\le \gamma_0*p_{2 \delta}(t_1-t_2)\le \gamma_0(t_1-t_2)\,,
\end{equation*}
where the last estimate follows by applying Fourier transform.
In the same way, we also obtain
\begin{equation*}
  \int_{\RR^d}\eta_\varepsilon(x_1,z)\eta_{\varepsilon}(x_2,z) dz\le \gamma*p_{2\varepsilon}(x_1-x_2)\le \gamma(x_1-x_2)\,.
\end{equation*}
Thus, we have
\begin{equation}
  \label{e.uniform-Q-bound} 
  Q_{\varepsilon,\delta}(t_1,t_2,x_1,x_2)\le \gamma_0(t_1-t_2)\gamma(x_1-x_2)\,.
\end{equation}
We shall assume throughout the remaining part of the paper  that the initial condition $u_0(x)$ is bounded, nonnegative and non trivial. 
In such case, $p_t*u_0(x)> 0$ for every $t>0$ and $x\in\RR^d$. 
Let $B$ be a standard Brownian motion starting at $x$. We denote
\begin{align*}
  V_t(B,W_{\varepsilon,\delta})
  &=\int_0^t \dot{W}_{\varepsilon,\delta}(t-s,B_s)ds-\frac12\int_0^t\int_0^t Q_{\varepsilon,\delta}(t-s,t-r,B_s,B_r)dsdr
\end{align*}
and
\[
\Theta(B, W_{\vare, \de})=\Theta_{t,x}(B,W_{\varepsilon,\delta})= u_0(B_t) e^{V_t(B,W_{\varepsilon,\delta})}\,.
\]
We define the random field 
\begin{equation}\label{eqn:FKtheta}
  u_{\vare, \de}(t,x)=u_{\vare, \de}(t,x,W_{\vare, \de})=\EE^B  \left[ \Theta(B, W_{\vare, \de})\right]\,,
\end{equation}
where $\EE^B$ denotes the expectation with respect to the Brownian
motion while   $W_{\vare, \de}$ is considered as ``constant".
To stress the dependence of $u_{\varepsilon,\delta}$ on $\dot W_{\vare, \de}$, we write $u_{\vare, \de }(t,x,W_{\vare, \de})$.
From the Feynman-Kac formula (see \cites{HHNT,HNS} and references therein) we see that $u_{\varepsilon,\delta}$ is the random field solution to the equation   
\begin{equation}
\begin{cases}
\frac{\partial u_{\vare, \de}(t,x)}{\partial t} =\frac12 \Delta u_{\vare, \de}(t,x)+ u_{\vare, \de}(t,x)\diamond \dot W_{\vare, \de}(t,x)\\
u_{\vare, \de}(0, x) =u_0(x) \quad \hbox{is given}\,.
\end{cases}
\label{e.approximated-spde-add} 
\end{equation}
Although we shall not use this fact in what follows.
The following result can be obtained analogously as \cite{HHNT}*{Theorem 3.6}.
\begin{proposition}\label{p.5.1}  
 Let $\left\{ \WW(t,x), t\in\RR, x\in \RR^d\right\}$ be a space time Brownian field. Let $\dot  W_{\vare, \de}$  be defined  by \eqref{e.appr-w} and let $\ga_0,\gamma$ satisfy Hypotheses  \ref{h2} and \ref{hspa}. 
 Assume that the initial condition $u_0$ is nonnegative  and bounded.
  Then, we have the following statements. 

(i) As $\vare, \de\rightarrow 0$ the approximated solution $u_{\vare, \de }(t,x)$ defined by \eqref{eqn:FKtheta}
converges to $u (t,x)$ in $L^p(\Om, \cF, P)$ 
for any $p\in [1, \infty)$,  where 
 $u (t,x)$ is the mild solution of \eqref{e.1.1} with $\its=1$.
  
(ii) For any positive integer $k$,  we have 
\begin{equation*}
\EE \left[u_{\vare, \de} (t,x)\right]^k
=\EE^{B^1, \cdots, B^k}
\Bigg[ \prod_{j=1}^k u_0(x+B^j_t) \exp\Bigg\{\sum_{1\le i<j\le k} \int_0^t\int_0^t Q_{\vare, \de}
(s,r, x+B_{t-s}^i, x+B_{t-r}^j) drds\Bigg\}\Bigg]\,,
\end{equation*}
where $B^1, \cdots, B^k$ are independent Brownian motions. 

(iii)  For any $p\ge 1$, we have
\begin{eqnarray}
\sup_{\vare, \de>0} \EE \left|u_{\vare, \de} (t,x)\right|^p
\le C  \exp\left\{ c p^{\frac{4- \alpha}{2- \alpha}} t^\beta \right\}\,,
\label{e.approxu-bound} 
\end{eqnarray}
where $C$ and $c$ are two positive constants independent of
$\vare$ and $\de$.
\end{proposition}
\begin{proof} The proofs  of  parts (i) and (ii)   are similar to those of the corresponding results in \cite{HHNT}.      The uniform bound in part (iii) follows from  the uniform bound \eqref{e.uniform-Q-bound}.
\end{proof}

From now on we assume that the initial condition $u_0$ is non-negative and non-trivial.  
Let $W_{\vare, \de} $ be  any (fixed) sample path of the approximated Gaussian noise process and now we consider $\Theta(B, W_{\vare, \de})$ as a functional of the Brownian motion $B$.
For a  measurable functional $F$ of Brownian motion $B$, we define its weighted expectation as follows.
\begin{align}\label{def:EBW}
 \EE_{W_{\vare, \de}}^B \left[F\right]
= \frac{\EE^B\left[
F\Theta(B, W_{\vare, \de})\right]  }
{ \EE^B\left[\Theta(B, W_{\vare, \de}) \right]}
= \frac{\EE^B\left[
F\Theta(B, W_{\vare, \de}) \right]}
{ u_{\vare, \de}(t,x, W_{\vare, \de})}\,.
\end{align}
Let $B^1,B^2$ be two independent identical copies of $B$. We shall  use the notation
\begin{align}
 \EE_{W_{\vare, \de}}^{B^1, B^2}  \left[F\right]
&= \frac{\EE^{B^1, B^2} \left[
F\Theta(B^1, W_{\vare, \de})\Theta(B^2, W_{\vare, \de}) \right]  }
{  \EE^{B^1, B^2} \left[\Theta(B^1, W_{\vare, \de}) \Theta(B^1, W_{\vare, \de}) \right]}
= \frac{\EE^{B^1, B^2} \left[
F\Theta(B^1, W_{\vare, \de})\Theta(B^2, W_{\vare, \de}) \right] }
{ |u_{\vare, \de} (t,x, W_{\vare, \de})|^2}\,, 
\end{align}
where $\EE^{B^1, B^2} $ denotes the expectation with respect to the Brownian
motions $B^1,B^2$ while  $W_{\vare, \de}$ is considered as ``constant". 

Clearly, for any fixed sample path $W_{\vare, \de}$,  $\EE_{W_{\vare, \de}}^B$ and $\EE_{W_{\varepsilon,\delta}}^{B^1,B^2}$ satisfy the properties of  expectation: for example, we have H\"older inequality, Jensen inequality, etc.  
 
Let $\dot{L}^2(\RR\times\RR^d) $ be the space of all Lebesgue measurable function f in $L^2(\RR\times\RR^d)$ such that the weak derivative $\dot{f}:=\partial^{d+1}_{t,x_1,\dots,x_d}f$ exists and belongs to $L^2(\RR\times\RR^d)$. We define the following pseudometric on $\dot{L}^2_d$
\begin{equation}\label{e.def-distance}
   d( f, g)=\|\dot{f}-\dot{g}\|_{L^2(\RR^{d+1})}\quad\textrm{for every }f,g\in \dot{L}^2_d \,.
\end{equation} 
For each $\lambda>0$, $(t,x)\in\RR\times\RR^d$, we define the event
\begin{equation}\label{def:Alambda}
A_\la(t,x)
=\Bigg\{ \WW_{\vare, \de}'\in \dot{L}^2_d:\  
 u_{\varepsilon,\delta}(t,x, W_{\vare, \de}')\ge \frac12 p_t*u_0(x)  \quad\textrm{and}\quad  
  \int_0^t\int_0^t      \EE^{B^1, B^2} _{W_{\vare, \de}'}
 Q   (t-r, t-s, B_r^1, B_s^2)  dsdr  
\le \la\Bigg\} \,. 
\end{equation}
We note that since $u_{\varepsilon,\delta}(t,x)$ is a functional of the Gaussian process $\WW_{\varepsilon,\delta}$, $A_\lambda(t,x)$ is in fact a random subset of $\dot{L}^2_d$. In addition, $P \left( A_\lambda(t,x)\right)$ is the probability of the event $\{\1_{A_\lambda(t,x)}(\WW_{\varepsilon,\delta})=1 \}$. For each $t>0$, we define the constant
\begin{equation}\label{def:lambdat}
  \lambda(t)=\sup_{x\in\RR^d} 32\frac{[\EE u^2(t,x)] [\EE u^2_{(\sqrt2)}(t,x)]}{|p_t*u_0(x)|^4}\,,
\end{equation}
where $u_{(\sqrt2)}$ is the solution of \eqref{e.1.1} with $\its=\sqrt 2$. We observe that from Theorem \ref{t.hhnt}, $\lambda(t)\le C e^{ct^\beta}$ for some finite constants $C,c>0$.
\begin{lemma}\label{lem:uabove}
  For every $\lambda>0$, for every sample path $\WW_{\vare, \de}'$ in $A_\lambda(t,x)$,
  \begin{equation}
    \label{e.u-w-d-w-wprime} 
    u_{\varepsilon,\delta}(t,x,W_{\varepsilon,\delta})\ge \frac12 p_t*u_0(x) e^{-\sqrt \lambda\|\dot{\WW}_{\varepsilon,\delta}- \dot{\WW}'_{\varepsilon,\delta}\|_{L^2(\RR^{d+1})} }  \,.
  \end{equation}  
\end{lemma}
\begin{proof}
  By combining the definition of $\EE_{W_{\vare, \de}'}^B$ in \eqref{def:EBW}, Feynman-Kac formula \eqref{eqn:FKtheta} and Jensen's inequality, we see that
  \begin{align}
  u_{\vare, \de}(t,x, W_{\vare, \de}) 
    &=u_{\vare, \de}(t,x, W_{\vare, \de}')   \EE_{W_{\vare, \de}'}
    ^B \exp\left\{ \int_0^t \left[\dot W_{\vare, \de}-\dot
    W_{\vare, \de}'\right](t-s, B_s) ds \right\}
    \nonumber\\&\ge  u_{\vare, \de}(t,x, W_{\vare, \de}')
    \exp\left\{ \EE_{W_{\vare, \de}'}
    ^B \int_0^t \left[\dot W_{\vare, \de}-\dot
    W_{\vare, \de}'\right](t-s, B_s) ds \right\}\,. 
    \label{e.u-w-wprime-bound} 
  \end{align}
  We are now going to bound the weighted expectation appeared in the above exponential, which we denote by 
  \begin{align*}
    I := \EE_{W_{\vare, \de}'}
    ^B \int_0^t \left[\dot W_{\vare, \de}-\dot
    W_{\vare, \de}'\right](t-s, B_s) ds\,.
  \end{align*}
  From the H\"older inequality we have
  \begin{align*}
  I 
  &= \int_\RR \int_{\RR^d}     
  \EE_{W_{\vare, \de}'}
    ^B\int_0^t \eta_0(t-s-r)\eta(B_s-y)  ds \left[ \dot  \WW_{\vare, \de}(r, y)
    -\dot \WW_{\vare, \de}'(  r,y)\right] drdy \\
    &\le\left\{\int_\RR\int_{\RR^d} 
    \left(\EE_{W_{\vare, \de}'}
    ^B  \int_0^t \eta_0(t-s-r)\eta(B_s-y)ds\right)^2 drdy \right\}^{\frac12} \left\{\int_\RR\int_{\RR^d}  
     \left[ \dot  \WW_{\vare, \de}(r, y)
    -\dot \WW_{\vare, \de}'(  r,y)\right] ^2drdy\right\}^{\frac12} \\
    &=\left\{\EE_{W_{\vare, \de}'}
    ^{B^1, B^2}   \int_0^t\int_0^t 
    Q(t-r, t-\bar r, B_r^1, B_{\bar r}^2)  drd\bar r  \right\}^{\frac12} \|\dot{\WW}_{\varepsilon,\delta}-\dot{\WW}'_{\varepsilon,\delta}\|_{L^2(\RR^{d+1})}  \,. 
  \end{align*}
  Now if $W_{\vare,\de}'\in A_\la$, using the second inequality in \eqref{def:Alambda}, we obtain from the above that
  \begin{equation*}
    I\le \la^{1/2}\|\dot{\WW}_{\varepsilon,\delta}-\dot{\WW}'_{\varepsilon,\delta}\|_{L^2(\RR^{d+1})}  \,. 
  \end{equation*}
  Combining this with  \eqref{e.u-w-wprime-bound} and the first inequality in \eqref{def:Alambda} yields \eqref{e.u-w-d-w-wprime}.
\end{proof}
\begin{lemma}\label{l.A-lambda}  For every $(t,x)\in\RR_+\times\RR^d$, we have
  \begin{equation}\label{est:PAla}
    P(A_{\la(t)}(t,x)) \ge  \frac18 \frac{|p_t*u_0(x)|^2}{\EE u^2(t,x)} \,.
  \end{equation}
  In particular, there are constants $C,c>0$ such that
  \begin{equation}\label{e.prob-A-lambda} 
    P(A_{\la(t)}(t,x)) \ge b(t)\ge C e^{-c t^\beta}\,,
  \end{equation}
  where
  \begin{equation}\label{def:bt}
    b(t)=\left(8 \sup_{x\in\RR^d}\frac{\EE u^2(t,x)}{|p_t*u_0(x)|^2} \right) ^{-1}\,.
  \end{equation}
\end{lemma} 
\begin{proof}
  From the definition of $\EE^{B^1, B^2} _{W_{\vare, \de}'}$ we see that $A_\lambda(t,x)$ contains paths $\WW'_{\varepsilon,\delta}$ such that $u_{\varepsilon,\delta}(t,x,W'_{\varepsilon,\delta})\ge\frac12p_t*u_0(x) $ and
  \begin{align*}
      \EE^{B^1, B^2}\Bigg[ \int_0^t\int_0^t Q(t-r, t-s,  B_r^1,  B_s^2) dsdr\  \Theta(B^1, W_{\vare, \de}' ) \Theta(B^2, W_{\vare, \de}' )  \Bigg]  \le \frac \lambda 4p_t*u_0(x)\,.
  \end{align*}
  It follows that
  \begin{equation*}
    P(A_\la(t,x))
  \ge   P\left(u_{\varepsilon,\delta}(t,x,   W_{\varepsilon,\delta}')\ge\frac12 p_t*u_0(x) \right)-
    \xi_{n, \la}(t,x)\,,
  \end{equation*}
  where 
  \begin{align*}
  \xi_{n, \la}(t,x)
  =P   \Bigg\{  \EE^{B^1, B^2}\Bigg[ \int_0^t\int_0^t Q(t-r, t-s,  B_r^1,  B_s^2) dsdr  \Theta(B^1, W_{\vare, \de}' ) \Theta(B^2, W_{\vare, \de}' )  \Bigg] >\frac \lambda 4 p_t*u_0(x)
  \Bigg\}\,. 
  \end{align*}
  The first probability can be estimated by the Paley-Zygmund's inequality,
  \begin{align*}
    P\left(u_{\varepsilon,\delta}(t,x,   W_{\varepsilon,\delta})\ge \frac12p_t*u_0(x) \right)
    &\ge \frac14 \frac{|p_t*u_0(x)|^2}{\EE u_{\varepsilon,\delta}^2(t,x)}
    \ge \frac14 \frac{|p_t*u_0(x)|^2}{\EE u^2(t,x)}\,.
  \end{align*} 

  We  use the Chebyshev inequality to estimate $\xi_{n, \la}(t,x)$.  By Fubini's theorem, we have
   \begin{align*}
   I&:=\EE^{W_{\vare,\de}'} \EE^{B^1, B^2}\lt[ \int_0^t\int_0^t Q(t-r, t-s,  B_r^1,  B_s^2) dsdr  \Theta(B^1, W_{\vare, \de}' ) \Theta(B^2, W_{\vare, \de}' )  \rt]
    \\&=      \EE^{B^1, B^2} 
    \lt[ \int_0^t\int_0^t   Q  (t-r, t-s,  B_r^1,  B_s^2) dsdr\EE^{W_{\vare,\de}'} \lt[  \Theta(B^1, W_{\vare, \de}' )
    \Theta(B^2, W_{\vare, \de}' ) \rt]  \rt]
    \\&=  \EE^{B^1, B^2}  \lt[   u_0( B_t^1)u_0( B_t^2) \int_0^t\int_0^t Q(t-r, t-s,  B_r^1,  B_s^2) dsdr    
    \exp \left\{\int_0^t\int_0^t  
      Q_{\vare, \de} (t-r, t-s,  B_r^1,  B_s^2)  drds
    \right\}  \rt]\,.
  \end{align*}
  Using \eqref{e.uniform-Q-bound} and the inequality $z\le e^{|z|}$, we get
  \begin{align*}
    I&\le\EE^{B^1,B^2}\lt[   u_0( B_t^1)u_0( B_t^2) \exp \left\{2\int_0^t\int_0^t  
      Q (t-r, t-s,  B_r^1,  B_s^2)  drds
    \right\}  \rt]=\EE \lt[|u_{\sqrt 2}(t,x)|^2\rt]\,,
  \end{align*}
  where $u_{\sqrt 2}$ is the solution to \eqref{e.1.1} with $\its=\sqrt 2$.
  Thus, by Chebyshev inequality, we have
  \[
  \xi_{n, \la}(t,x) \le \frac4 \lambda\frac{I_1}{ |p_t*u_0(x)|^2 }
  \le \frac 4 \lambda\frac{\EE u_{\sqrt 2}^2(t,x)}{|p_t*u_0(x)|^2}\,.
  \]
  Therefore, we have 
  \begin{align*}
  P(A_\la(t,x)) \ge \frac14 \frac{|p_t*u_0(x)|^2}{\EE u^2(t,x)}- \frac 4 \lambda\frac{\EE u_{\sqrt 2}^2(t,x)}{|p_t*u_0(x)|^2}
  \ge  \frac18 \frac{|p_t*u_0(x)|^2}{\EE u^2(t,x)}\,. 
  \end{align*}
  for every $\lambda\ge 32\frac{\EE u^2(t,x) \EE u^2_{\sqrt2}(t,x)}{|p_t*u_0(x)|^4}$. This implies the estimate \eqref{est:PAla}. The estimate \eqref{e.prob-A-lambda} is derived from \eqref{est:PAla} and Theorem \ref{t.hhnt}.
\end{proof}

As in \cite{flores}, we shall use concentration inequality to bound the probability associated with $\|\dot{\WW}_{\varepsilon,\delta}- \dot{\WW}'_{\varepsilon,\delta}\|_{L^2(\RR^{d+1})}$. 
\begin{lemma}\label{lem:dconcentration} For every $a>0$ we have
  \begin{equation}
    P \left( d(\WW_{\varepsilon,\delta},A_\lambda(t,x))>a +2\sqrt{\log\frac 2{b(t)}} \right) \le 2e^{-\frac{a^2}4}\,.
  \end{equation}
\end{lemma}
\begin{proof} Step 1: We begin with some notation and analysis of the noise.  The covariance of $\WW_{\vare, \de}$ is 
  \[
  \tilde Q_{\vare, \de}(s, t, x, y)=\EE \left[\dot \WW_{\vare, \de}(s,x)
  \dot \WW_{\vare, \de}(t,y)\right]=e^{-\frac{\vare}{2}(|x|^2+|y|^2)-\frac \delta2(|s|^2+|t|^2)}p_{2\de}(t-s)p_{2\vare}(x-y) \,.
  \]
  It follows that that for any fixed $\vare, \de$, the quantity $\int_\RR\int_\RR\int_{\RR^d}\int_{\RR^d}|\tilde Q_{\varepsilon,\delta}(s,t,x,y)|^2dsdtdxdy$ is finite. In particular, $\tilde Q_{\varepsilon,\delta}$ defines a Hilbert-Schmidt operator on $L^2(\RR\times\RR^d)$.
  By Mercer's theorem,  
  there is an  orthonormal  basis  $\left\{e_k: k=1, 2, \cdots\right\}
   $ of  $L^2(\RR\times \RR^d)$ such that 
  \begin{equation}
  \tilde Q_{\vare, \de} (r,s,y,z)=\sum_{k=1}^\infty \la_k e_k(r,y)e_k(s,z)\quad\textrm{in }L^2(\RR\times\RR^d) \,, 
  \end{equation}
  where  $\la_k\ge 0$ and $\sum_{k=1}^\infty \la_k^2<\infty$. 
  Obviously for each $k\ge1$, $\lambda_k,e_k$ may depend on $\varepsilon,\delta$, however, this dependence will be omitted in the notation. 
  With this basis, every function $f\in \dot{L} ^2(\RR\times\RR^d)$ has an expansion of the form $\dot{f}=\sum_{k=1}^\infty \sqrt{\lambda_k}e_k(f) e_k$,
  where $e_k(f)=\lambda_k^{-1/2} \langle \dot{f},e_k \rangle_{L^2(\RR\times\RR^d)}$. For later purposes, we also denote $\dot{f}_n=\sum_{k=1}^n \sqrt{\lambda_k}e_k(f) e_k$. In particular, we have the following  Karhunen-Lo\`eve  expansion for $\dot \WW_{\vare, \de} (r,y)$
  \begin{equation*}
  \dot \WW_{\vare, \de} (r,y)=\sum_{k=1}^\infty \sqrt{\la_k}  e_k(r,y) G_k\,,
  \end{equation*}
  where $G_k=\lambda_k^{-1/2}\langle\WW_{\varepsilon,\delta},e_k \rangle_{L^2(\RR\times\RR^d)}$. The law of $\WW_{\varepsilon,\delta}$ implies that $G_1,G_2,\dots$ are independent standard normal random variables.
  We also denote
  \begin{equation*}
  \dot \WW_{\vare, \de} '(r,y)=\sum_{k=1}^\infty \sqrt{\la_k}  e_k(r,y) G_k'\,, 
  \end{equation*}
  where $\{G_k':k=1,2,\dots\} $ is a sample of  $\left\{ G_k: k=1, 2, \cdots  \right\}$ in the same stochastic basis.
  Following the previous notation, we write
  \begin{equation}
    \dot \WW_{\vare, \de,n} (r,y) = \sum_{k=1}^n \sqrt{\la_k}  e_k(r,y) G_k
    \quad\textrm{and}\quad \dot \WW_{\vare, \de, n} '(r,y) = \sum_{k=1}^n \sqrt{\la_k}  e_k(r,y) G_k'\,.
  \end{equation}
  For every $h$ in $L^2(\RR\times\RR^d)$ such that $\|h\|_{L^2(\RR\times\RR^d)}=1$, we apply Cauchy-Schwarz inequality and Young inequality to see that
  \begin{align*}
    &\int_{\RR\times\RR^d}\int_{\RR\times\RR^d} \tilde Q_{\varepsilon,\delta}(s,t,x,y)h(s,x)h(t,y)dsdxdtdy
    \\&\le \int_{\RR\times\RR^d}\int_{\RR\times\RR^d}p_{2 \delta}(t-s)p_{2 \varepsilon}(x-y)|h(s,x)||h(t,y)|dsdxdtdy
    \\&\le \|h\|_{L^2(\RR\times\RR^d)} \|(p_{2 \delta}\otimes p_{2 \varepsilon})* h\|_{L^2(\RR\times\RR^d)}
    \le \|h\|_{L^2(\RR\times\RR^d)}^2\|p_{2 \delta}\otimes p_{2 \varepsilon}\|_{L^1(\RR\times\RR^d)}=1\,.
  \end{align*}
  This implies that $\sup_k \lambda_k\le 1$. In particular, for every $n$, we have 
  \begin{align}\label{est:ddn}
    |d(\WW_{\varepsilon,\delta,n},\WW_{\varepsilon,\delta,n}')|^2
    =\sum_{k=1}^n \lambda_k (G_k-G_k')^2
    \le\sum_{k=1}^n (G_k-G_k')^2=: |d_n(G^{(n)},G'^{(n)})|^2\,,
  \end{align}
  where $d$ is defined in \eqref{e.def-distance}, $G^{(n)}=(G_1,\dots,G_n)$, $G'^{(n)}=(G'_1,\dots,G'_n)$ and $d_n$ is the Euclidean distance in $\RR^n$. 

  Step 2:  For each $a>0$, $\rho\in(0,1)$ and $n\ge1$, we denote by $K_{n,a,\rho}$ the set $\{f\in \dot{L}^2(\RR\times\RR^d) : d(f,f_n)<(1- \rho)a \}$. Using triangle inequality, we have
  \begin{align*}
    d(\WW_{\varepsilon,\delta},A_{\lambda})\le d(\WW_{\varepsilon,\delta},A_{\lambda}\cap K_{n,a,\rho})
    \le d(\WW_{\varepsilon,\delta,n},A_{\lambda}\cap K_{n,a,\rho})+ d(\WW_{\varepsilon,\delta},\WW_{\varepsilon,\delta,n})\,.
  \end{align*}
  In addition, for each $\WW'_{\varepsilon,\delta}$ in $A_{\lambda}\cap K_{n,a,\rho}$, using \eqref{est:ddn}, we have
  \begin{align*}
    d(\WW_{\varepsilon,\delta,n},\WW'_{\varepsilon,\delta})
    \le d(\WW_{\varepsilon,\delta,n},\WW'_{\varepsilon,\delta,n})+d(\WW'_{\varepsilon,\delta,n},\WW'_{\varepsilon,\delta})
    \le d_n(G^{(n)},B_{\lambda,n})+(1- \rho)a \,,
  \end{align*}
  where 
  \[
  B_{\la, n}=\left\{ (G_1', \cdots, G_n')\in\RR^n\,: \
  \sum_{k=1}^\infty  \sqrt{\la_k}  e_k  G_k'\in A_\la\right\}\,.
  \]
  It follows that
  \begin{equation*}
    d(\WW_{\varepsilon,\delta},A_\lambda)\le d_n(G^{(n)},B_{\lambda,n})+d(\WW_{\varepsilon,\delta},\WW_{\varepsilon,\delta,n})+(1- \rho)a \,.
  \end{equation*}
  Hence, 
  \begin{align}
    P\lt(d(\WW_{\varepsilon,\delta},A_\lambda)>a\rt)
    &\le P \left( d_n(G^{(n)},B_{\lambda,n})+d(\WW_{\varepsilon,\delta},\WW_{\epsilon,\delta,n})> \rho a \right) 
    \nonumber\\&\le P \left( d_n(G^{(n)},B_{\lambda,n})>\rho^2 a\right)
    +P \left( d(\WW_{\varepsilon,\delta},\WW_{\varepsilon,\delta,n})>\rho(1- \rho)a \right)  \,.
    \label{tmp:PGdn}
  \end{align}
  Let us estimate the first probability on the right-hand side above. Clearly,  $P(B_{\la, n})\ge P(A_\la)$ for any finite $n$, hence,  from \eqref{e.prob-A-lambda}, we have
  \begin{equation}
  P(B_{\la(t), n})\ge b(t)\,.
  \label{e.prob-B-lambda-n}
  \end{equation}
  From \eqref{e.prob-B-lambda-n}
  and from   the concentration inequality (see \cite{talagrand}*{ Lemma 2.2.11}), it follows  that
  \begin{align*}
    P \left( d_n(G^{(n)},B_{\lambda,n})>\rho^2a \right)\le 2\exp\lt\{-\frac14\lt(\rho^2a-2\sqrt{\log\frac2{b(t)}}  \rt)^2 \rt\}
  \end{align*}
  for every $a>\rho^{-2} 2\sqrt{\log\frac2{b(t)}}$.
  The last probability in \eqref{tmp:PGdn} can be estimated by Chebyshev  inequality,
  \begin{align*}
  P\left(d(\WW_{\vare, \de}, \WW_{\vare, \de,n})\ge \rho(1- \rho)a
  \right)
  &\le [\rho(1- \rho) a]^{-2}  \EE d^2(\WW_{\vare, \de}, \WW_{\vare, \de,n})
  \le  [\rho(1- \rho) a]^{-2}\sum_{k=n+1}^\infty \la_k \,.
  \end{align*}
  Hence, we derive from \eqref{tmp:PGdn} that
  \begin{align*}
    P \left( d(\WW_{\varepsilon,\delta},A_\lambda)>a \right)\le  
    2\exp\lt\{-\frac14\lt(\rho^2 a-2\sqrt{\log\frac2{b(t)}} \rt)^2 \rt\}
    + [\rho(1- \rho) a]^{-2}\sum_{k=n+1}^\infty \lambda_k\,.
  \end{align*}
  Since $\int_\RR \int_{\RR^{d}}| \tilde Q_{\vare, \de}(s, s, x, x) |  dsdx$ is finite, $\sum_{k=1}^\infty \lambda_k$ is also finite. Hence, by sending $n$ to infinity and $\rho$ to 1 in the above estimate, we obtain the result.
\end{proof}
\begin{theorem}\label{thm:usmallball} 
	Assume that Hypotheses \ref{h2}, \ref{hspa} and \ref{hroot} hold.
  Let $t>0$ and $x\in\RR^d$ be fixed.
  For every $0<r<\frac12e^{-2\sqrt{\lambda(t)\log\frac2{b(t)}}}$, we have
  \begin{equation}\label{est:Prob.uleftr}
     P \left( \frac{u(t,x)}{p_t*u_0(x)}<r \right)\le 2\exp\lt\{-\frac14\lt(\frac{\log(2r)}{\sqrt{\lambda(t)} }+2\sqrt{\log\frac2{b(t)}} \rt)^2 \rt\} \,.
  \end{equation} 
  Written another way, for every $a>0$, we have
  \begin{equation}\label{est:Prob.ulefta}
    P \left( \frac{u(t,x)}{p_t*u_0(x)}<\frac12 e^{-2\sqrt{\lambda(t)\log\frac2{b(t)}}} e^{-\sqrt{\lambda(t)}a} \right)\le 2e^{-\frac{a^2}4}\,.
  \end{equation}
  Let $\lambda_t=\tilde Ce^{\frac{\tilde c}2 t^\beta}$ where $\tilde C,\tilde c$ are the constants in Lemma \ref{l.A-lambda}. There are positive constants $r_0$, $C$, $c_1$,  $c_2$, $c_3$ and $c_4>0$  such that  for every $t>0,x\in\RR^d$ and for every $0<r<\frac12\exp\{-c_4 \sqrt{\lambda_t}\} $,
\begin{equation}
P\lt(\frac{u (t,x)}{p_t*u_0(x)} \le r\rt) 
 \le  C \exp\lt\{ - \lt[  c_1  \exp\left(-c_3  t^\beta \right) \log r  -c_2\sqrt{1+t^\beta} \rt]^2\rt\}\,. 
\label{e.upper-left-bound-general}
\end{equation}
\end{theorem} 
\begin{proof}
  Applying Lemma \ref{lem:uabove}, we have
  \begin{align*}
  P\lt(\frac{u_{\varepsilon,\delta}(t,x)}{p_t*u_0(x)}\le r\rt)
  \le P\lt(d(\WW_{\vare, \de}, A_{\la(t)}(t,x))\ge \frac{-\log(2r)}{\sqrt \la(t)}\rt)\,. 
  \end{align*}
  Due to the range of $r$, we can choose $a>0$ such that $a+2\sqrt{\log\frac2{b(t)}}=-\frac{\log(2r)}{\sqrt{\lambda(t)}} $ and apply Lemma \ref{lem:dconcentration} to obtain \eqref{est:Prob.uleftr} with $u$ being replaced by $u_{\varepsilon,\delta}$. By sending $\varepsilon,\delta$ to $0$ and applying Proposition \ref{p.5.1} we obtain \eqref{est:Prob.uleftr}. 
\end{proof}

\begin{remark}  In the case when $\dot{W}$ is a space time white noise with spatial dimension one (that is $d=\alpha_0=\alpha=1$), the above theorem yields
\begin{equation*}
P\lt(\frac{u (t,x)}{p_t*u_0(x)} \le r\rt) 
 \le  C \exp\left\{ - \left(  c_1 \exp\left(-c_3 t\right) \log r +c_2 \sqrt {1+t}   \right)^2\right\}\,,
\end{equation*}
for sufficiently small $r$. This implies the result of \cite{flores}. 
\end{remark} 
\begin{corollary}\label{cor:u-neg-m}
  For every $p>0$ and every $(t,x)\in \RR_+\times\RR^d$, we have
  \begin{equation}
    \EE|u(t,x)|^{-p}\le 2^{p}e^{2p\sqrt{\lambda(t)\log\frac2{b(t)}}}\lt(1+4\sqrt{\pi p^2 \lambda(t)}e^{p^2 \lambda(t)}\rt)|p_t*u_0(x)|^{-p} \,.
  \end{equation}
\end{corollary}
\begin{proof}
  Set $F=\frac{u(t,x)}{p_t*u_0(x)}2e^{2\sqrt{\lambda(t)\log\frac2{b(t)}}}$. We observe that
  \begin{align*}
    \EE F^{-p}=p\int_0^\infty r^{-p}P(F<r)\frac{dr}r
    &\le p\int_0^1 r^{-p}P(F<r)\frac{dr}r+p\int_1^\infty r^{-p}\frac{dr}r
    \\&= p\sqrt{\lambda(t)} \int_0^\infty e^{p\sqrt{\lambda(t)}a}P\lt(F<e^{-\sqrt{\lambda(t)}a}\rt)da+1\,.
  \end{align*}
  The above integral can be estimated using Theorem \ref{thm:usmallball},
  \begin{align*}
    p\sqrt{\lambda(t)} \int_0^\infty e^{p\sqrt{\lambda(t)}a}P\lt(F<e^{-\sqrt{\lambda(t)}a}\rt)da
    \le 2p\sqrt{\lambda(t)}\int_0^\infty e^{p\sqrt{\lambda(t)}a}e^{-\frac{a^2}4} da
    \le 4\sqrt{\pi p^2\lambda(t)}e^{p^2 \lambda(t)}\,.
  \end{align*}
  These estimates yield the result.
\end{proof}

\section{Non-degeneracy of Malliavin derivatives} 
\label{sec:non_degeneracy_of_malliavin_derivatives}
  Corollary \ref{cor.Dku-bound} shows that the Malliavin derivatives  of all orders of $u$ exist. In this section, we show that $\|Du(t,x) \|_\cH$ has finite negative moments of all orders. 
 Let us explain the general idea. We consider the approximation scheme in Section \ref{sec:left_tail} and derive a Feynman-Kac formula for $\|Du_{\varepsilon,\delta}(t,x)\|_\cH^2$, which has analogous structure to \eqref{eqn:FKtheta}. 
  This makes it suitable to apply the methods of Section \ref{sec:left_tail} to estimate the small ball probability of $\|Du(t,x)\|_\cH^2$. As we have seen in Corollary \ref{cor:u-neg-m}, small ball probability estimate yields finiteness of negative moments.   

  We continue with the setup in Section \ref{sec:left_tail} and the approximation scheme used there. 
Let $B$ denote a standard Brownian   motion in $\RR^d$ starting at $x$ and $B^1,B^2,\dots$ be its independent identical copies.
  For every $n\ge2$, we define
  \begin{equation}\label{def:Qn}
    Q^{(n)}_{\varepsilon,\delta}(t,B^1,\cdots,B^n) =\frac12\sum_{j,k=1;j\neq k}^n\int_0^t\int_0^t Q_{\varepsilon,\delta}(t-s,t-r,B^j_s,B^k_r)dsdr \,.
  \end{equation}
  and $Q^{(n)}(t,B^1,\cdots,B^n)$ with obvious modifications. We observe that $Q^{(n)}$ is at best a distribution, however, it has finite exponential moments. This is because $\EE^{B^1,\dots,B^n}e^{\lambda Q^{(n)}(t,B^1,\dots,B^n)}$ coincides with $\EE u_{(\sqrt\lambda)}(t,x)^n$, which is finite in view of Theorem \ref{t.hhnt}.
  Let us fix $t>0$ and $x\in\RR^d$ and define the following quantities
  \begin{equation}\label{def:lb}
    \lambda(t,x)=32\frac{\tilde I(t,x)\EE[\|Du(t,x)\|_\cH^4]}{\lt(\EE \|Du(t,x)\|_\cH^2\rt)^4 }
    \quad\textrm{and}\quad b(t,x)=\frac18 \frac{\lt(\EE \|Du(t,x)\|_\cH^2\rt)^2}{ \EE \|Du(t,x)\|_\cH^4}
    \,,
  \end{equation}
  where
  \begin{equation}\label{def.tiI}
     \tilde I(t,x)=4\EE^{B^1,\cdots, B^4}  \lt[ \prod_{j=1}^4 u_0(B^j_t)  Q^{(2)}(t, B^1,B^2)Q^{(2)}(t, B^3,B^4)  Q^{(2)}(t,B^1,B^3) e^{Q^{(4)}(t,B^1,\dots,B^4)}  \rt]\,.
  \end{equation} 
  Note that $\lambda(t,x)$ and $b(t,x)$ are positive and finite (by Corollary \ref{cor.Dku-bound}).
  \begin{theorem}\label{thm:Duneg} For every $p>0$ and $(t,x)\in\RR_+\times\RR^d$,
    \begin{equation}\label{est:Du.neg.moment}
      \EE\|Du(t,x)\|_\cH^{-2p}\le2^{p}e^{2p\sqrt{\lambda(t,x)\log\frac2{b(t,x)}}}\lt(1+4\sqrt{\pi p^2 \lambda(t,x)}e^{p^2 \lambda(t,x)}\rt)(\EE \|D u(t,x)\|_\cH^2)^{-p}\,.
    \end{equation}
  \end{theorem}

We recall that   $
  \dot{W} (t,x)=\int_\RR\int_{\RR^d} \eta_0(t-r)\eta(x-z)\WW(dr,dz)$. Thus, a functional of $W$ is also a functional
  of $\WW$. It is then more convenient to work with Malliavin derivative with respect to the white noise $\WW$, denoted by $D^\WW$. The following result relates the two Malliavin derivatives (chain rule).  
  \begin{proposition}\label{prop:DDW}
    Suppose $F$ is in $D^{1,2}(W)$ (wrt $W$). Then $F$ is a well-defined functional of $\WW$ and $F$ is in $D^{1,2}(\WW)$. Moreover, we have 
    \begin{equation}
      D^\WW F=(\eta_0\otimes \eta)\star DF\,,
      \label{e.two_der}
    \end{equation}
    where $\star$ is the convolution in $\RR^{d+1}$.
    In particular, $\|DF\|_\cH=\|D^\WW F\|_{L^2(\RR^{d+1}) } $.
  \end{proposition}
  \begin{proof}
    If $F=\int \phi dW$,  where $\phi\in \cH$ is  deterministic,  then $F=\int (\eta_0\otimes \eta)\star \phi d\WW $. So $DF= \phi$ and $D^\WW F=(\eta_0\otimes \eta)\star\phi=(\eta_0\otimes \eta)\star D F $, which verifies the relation \eqref{e.two_der}. Now let $F=\psi(F_1, \cdots, F_n)$,  where
    $\psi$ is a polynomial of $n$ variables and $F_i=
    \int \phi_i dW$,  where  $\phi_i\in \cH$ are   deterministic.
    Let $\partial _i\psi$ denote the partial derivative of $\psi$ with respect to the $i$-th coordinate. Using the chain rule and the previous argument, we see that
    \begin{align*}
      D^\WW F
      &=\sum_{i=1}^n \partial_i \psi(F_1,\cdots,F_n)D^\WW F_i
      =\sum_{i=1}^n \partial_i \psi(F_1,\cdots,F_n)(\eta_0\otimes \eta)\star D F_i
      \\&=(\eta_0\otimes \eta)\star\sum_{i=1}^n \partial_i \psi(F_1,\cdots,F_n) D F_i
      =(\eta_0\otimes \eta)\star DF\,.
    \end{align*}
    This verifies the relation \eqref{e.two_der}. Since the set of the functionals of
    the above form  is dense in $D^{1,2}(W)$, the complete result follows by a limiting argument. 
  \end{proof}
  Let us now derive a Feynman-Kac formula for $D^\WW u_{\varepsilon,\delta}(t,x)$. In view of \eqref{e.appr-w}, we can write
  \begin{align*}
    V_t(B,W_{\varepsilon,\delta})
    =\int_\RR\int_{\RR^d}\int_0^t \eta_{0,\delta}(t-s,r) \eta_{\varepsilon}(B_s,y)ds\WW(dr,dy)-\frac12\int_0^t\int_0^t Q_{\varepsilon,\delta}(t-s,t-r,B_s,B_r)dsdr\,,
  \end{align*}
  which implies
  \begin{align*}
    D^\WW_{r,y} V_t(B,W_{\varepsilon,\delta})=\int_0^t \eta_{0,\delta}(t-s,r) \eta_{\varepsilon}(B_s,y)ds\,.
  \end{align*}
  Hence, applying $D^\WW$ to \eqref{eqn:FKtheta}, we obtain
  \begin{equation}
    D^\WW_{r,y} u_{\varepsilon,\delta}(t,x)=\EE^B \lt[u_0(B_t)\lt(\int_0^t \eta_{0,\delta}(t-s,r) \eta_{\varepsilon}(B_s,y)ds\rt)e^{V_t(B,W_{\varepsilon,\delta})}\rt]\,.
  \end{equation}
  It follows from the above identity that
  \begin{equation}\label{eqn:FKDuep}
    \|D^\WW_\cdot u_{\varepsilon,\delta}(t,x)\|^2_{L^2(\RR^{d+1})}=\EE^{B^1,B^2}\lt[u_0(B^1_t)u_0(B^2_t) Q^{(2)}_{\varepsilon,\delta}(t,B^1,B^2)e^{V_t(B^1,W_{\varepsilon,\delta})+V_t(B^2,W_{\varepsilon,\delta})} \rt]\,.
  \end{equation}
  To simplify notation, we denote $Z_{\varepsilon,\delta}(t,x)=\|D^\WW u_{\varepsilon,\delta}(t,x)\|^2_{L^2(\RR^{d+1})} $ and $Z(t,x)=\|D^\WW u(t,x)\|^2_{L^2(\RR^{d+1})} $. From the above identities, we have
  \begin{equation}\label{eqn:Z1st}
    \EE Z_{\varepsilon,\delta}(t,x)=\EE^{B^1,B^2}\lt[u_0(B^1_t)u_0(B^2_t) Q^{(2)}_{\varepsilon,\delta}(t,B^1,B^2)e^{Q^{(2)}_{\varepsilon,\delta}(t,B^1,B^2)} \rt]
  \end{equation}
  and
  \begin{equation}\label{eqn:Z2nd}
    \EE Z^2_{\varepsilon,\delta}(t,x)
    =\EE^{B^1,\cdots,B^4}\lt[\prod_{j=1}^4 u_0(B^j_t) Q^{(2)}_{\varepsilon,\delta}(t,B^1,B^2)Q^{(2)}_{\varepsilon,\delta}(t,B^3,B^4)e^{Q^{(4)}_{\varepsilon,\delta}(t,B^1,\cdots,B^4)} \rt]\,.
  \end{equation}

  We now estimate the small ball probability of $Z_{\varepsilon,\delta}(t,x)$ using the methods in Section \ref{sec:left_tail}. We put 
  \[
    \Theta(B^1,B^2,W_{\varepsilon,\delta})=u_0(B^1_t)u_0(B^2_t) Q^{(2)}_{\varepsilon,\delta}(t,B^1,B^2) e^{V_t(B^1,W_{\varepsilon,\delta})+V_t(B^2,W_{\varepsilon,\delta})}\,.
  \]
  Let $W_{\varepsilon,\delta}$ be fixed. For every measurable functional $F$ of $B^1,B^2$ and every measurable functional $G$ of $B^1,\cdots,B^4$, we define their weighted expectations as follows 
  \begin{equation}
    \EE^{(2)}_{W_{\varepsilon,\delta}}[F]=\frac{\EE^B[F \Theta(B^1,B^2,W_{\varepsilon,\delta})]}{\EE^B[\Theta(B^1,B^2,W_{\varepsilon,\delta})]}
    =\frac{\EE^B[F \Theta(B^1,B^2,W_{\varepsilon,\delta})]}{Z_{\varepsilon,\delta}(t,x) }
  \end{equation}
  and
  \begin{equation}
     E^{(4)}_{W_{\varepsilon,\delta}}[G]=\frac{\EE^B[G \Theta(B^1,B^2,W_{\varepsilon,\delta}) \Theta(B^3,B^4,W_{\varepsilon,\delta})]}{\EE^B[\Theta(B^1,B^2,W_{\varepsilon,\delta}) \Theta(B^3,B^4,W_{\varepsilon,\delta})]}
    =\frac{\EE^B[G \Theta(B^1,B^2,W_{\varepsilon,\delta}) \Theta(B^3,B^4,W_{\varepsilon,\delta})]}{|Z_{\varepsilon,\delta}(t,x)|^2}\,.
  \end{equation}
  It is evident that these qualities also depend on $t,x$, however, this  dependence  will be omitted.  

  For each $\lambda>0$, $(t,x)\in\RR_+\times\RR^d$, we define
  \begin{equation}\label{def:tiAlambda}
    \tilde A_\la(t,x)
    =\Bigg\{ \WW_{\vare, \de}'\in \dot{L}^2_d:\  
    Z_{\varepsilon,\delta}(t,x, W_{\vare, \de}')\ge \frac12 \EE Z_{\varepsilon,\delta}(t,x)  
    \quad\textrm{and}\quad  
    \EE^{(4)} _{W_{\vare, \de}'}\sum_{j=1}^2\sum_{k=3}^4 Q^{(2)}   (t, B^j, B^k)
    \le \la\Bigg\} \,,
  \end{equation}
  \begin{equation}
    \lambda_{\varepsilon,\delta}(t,x)=32\frac{\tilde I(t,x)\EE[|Z_{\varepsilon,\delta}(t,x)|^2]}{|\EE Z_{\varepsilon,\delta}(t,x)|^4 }
    \quad\textrm{and}\quad b_{\varepsilon,\delta}(t,x)=\frac18 \frac{|\EE Z_{\varepsilon,\delta}(t,x)|^2}{ \EE |Z_{\varepsilon,\delta}(t,x)|^2}
    \,.
  \end{equation}
  We note that $\lambda_{\varepsilon,\delta}(t,x) $ and $b_{\varepsilon,\delta}(t,x) $ are positive finite constants for every $t,x$.
  \begin{lemma}\label{lem:Duphi}
    \begin{equation}
      Z_{\varepsilon,\delta}(t,x)\ge \frac12 \EE[Z_{\varepsilon,\delta}(t,x)]e^{-\sqrt \lambda d(\WW_{\varepsilon,\delta},\tilde A_\lambda)}
    \end{equation}
  \end{lemma}
  \begin{proof} Let $\WW'_{\varepsilon,\delta}$ be in $\tilde A_\lambda$. Using Jensen's inequality and the first inequality in \eqref{def:tiAlambda}, we see that
    \begin{align*}
      Z_{\varepsilon,\delta}(t,x,W)
      &=Z_{\varepsilon,\delta}(t,x,W') \EE^{(2)}_{W'_{\varepsilon,\delta}}\lt[ e^{\sum_{j=1}^2\int_0^t(\dot{W}_{\varepsilon,\delta}-\dot{W}'_{\varepsilon,\delta})(t-s,B^j_s)ds } \rt]
      \\&\ge Z_{\varepsilon,\delta}(t,x,W')\exp\lt\{ \EE^{(2)}_{W'_{\varepsilon,\delta}} \sum_{j=1}^2\int_0^t(\dot{W}_{\varepsilon,\delta}-\dot{W}'_{\varepsilon,\delta})(t-s,B^j_s)ds \rt\}
      \\&\ge \frac12\EE [Z_{\varepsilon,\delta}(t,x)]\exp\lt\{ \EE^{(2)}_{W'_{\varepsilon,\delta}} \sum_{j=1}^2\int_0^t(\dot{W}_{\varepsilon,\delta}-\dot{W}'_{\varepsilon,\delta})(t-s,B^j_s)ds \rt\}\,.
    \end{align*}
    The weighted expectation appeared in the above exponential can be estimated using Cauchy-Schwarz inequality,
    \begin{align*}
      &\EE^{(2)} _{W_{\vare, \de}'} \sum_{j=1}^2\int_0^t \left[\dot W_{\vare, \de}-\dot  W_{\vare, \de}'\right](t-s, B^j_s) ds
      \\&=\int_\RR \int_{\RR^d}  \EE_{W_{\vare, \de}'}
      ^{(2)}\sum_{j=1}^2\int_0^t \eta_0(t-s-r)\eta(B_s^j-y)  ds \left[ \dot  \WW_{\vare, \de}(r, y)
      -\dot \WW_{\vare, \de}'(  r,y)\right] drdy \\
      &\ge-\left\{\int_\RR\int_{\RR^d} 
      \left(\EE_{W_{\vare, \de}'}
      ^{(2)}  \sum_{j=1}^2\int_0^t \eta_0(t-s-r)\eta(B^j_s-y)ds\right)^2 drdy \right\}^{\frac12} \left\{\int_\RR\int_{\RR^d}  
       \left[ \dot  \WW_{\vare, \de}(r, y)
      -\dot \WW_{\vare, \de}'(  r,y)\right] ^2drdy\right\}^{\frac12} \\
      &=-\left\{\EE_{W_{\vare, \de}'}
      ^{(4)}  \sum_{j=1}^2\sum_{k=3}^4 Q^{(2)}(t,B^j,B^k) \right\}^{\frac12} \|\dot{\WW}_{\varepsilon,\delta}-\dot{\WW}'_{\varepsilon,\delta}\|_{L^2(\RR^{d+1})}  \,. 
    \end{align*}
    Now since $\WW_{\vare,\de}'\in \tilde A_\la$, using the second inequality in \eqref{def:tiAlambda}, we obtain from the above that
    \begin{equation*}
      \EE^{(2)} _{W_{\vare, \de}'} \sum_{j=1}^2\int_0^t \left[\dot W_{\vare, \de}-\dot  W_{\vare, \de}'\right](t-s, B^j_s) ds\ge- \la^{1/2}\|\dot{\WW}_{\varepsilon,\delta}-\dot{\WW}'_{\varepsilon,\delta}\|_{L^2(\RR^{d+1})}  \,. 
    \end{equation*}
    Combining these estimates yields
    \begin{equation*}
      Z_{\varepsilon,\delta}(t,x,W)\ge\frac12\EE[Z_{\varepsilon,\delta}(t,x)]e^{-\sqrt \lambda d(\WW_{\varepsilon,\delta},\WW_{\varepsilon,\delta}') }
    \end{equation*}
    for every $\WW'_{\varepsilon,\delta}\in\tilde A_\lambda$, which is the statement of the lemma. 
  \end{proof}
  \begin{lemma}\label{lem:tiA-lambda} Let $t>0$ and $x\in\RR^d$ be fixed but arbitrary. For every $\lambda$ satisfying
    \begin{equation}\label{eqn:ltia}
      \lambda\ge \lambda_{\varepsilon,\delta}(t,x) \,,
    \end{equation}
    we have
    \begin{equation}\label{est:tiA}
    P(\tilde A_\la(t,x)) \ge b_{\varepsilon,\delta}(t,x) \,.
    \end{equation}
  \end{lemma} 
  \begin{proof}
    Following the same arguments in the proof Lemma \ref{l.A-lambda}, we see that
    \begin{align*}
      P(\tilde A_\lambda(t,x))\ge P \left( Z_{\varepsilon,\delta}(t,x)\ge \frac12 \EE Z_{\varepsilon,\delta}(t,x) \right) -\tilde \xi_{n,\lambda}(t,x)\,,
    \end{align*}
    where
    \begin{align*}
    &\tilde\xi_{n, \la}(t,x)
    =P   \Bigg\{  \EE^{B^1,\cdots, B^4}\Bigg[\sum_{j=1}^2\sum_{k=3}^4 Q^{(2)}(t,B^j,B^k)  \Theta(B^1,B^2, W_{\vare, \de}' ) \Theta(B^3,B^4, W_{\vare, \de}' )  \Bigg] >\frac \lambda 4 |\EE Z_{\varepsilon,\delta}(t,x)|^2    \Bigg\}\,. 
    \end{align*}
    The first probability can be estimated by the Paley-Zygmund's inequality,
    \begin{align*}
      P\left(Z_{\varepsilon,\delta}(t,x)\ge \frac{1}{2}\EE Z_{\varepsilon,\delta}(t,x) \right)
    &\ge  
    \frac{|\EE Z_{\varepsilon,\delta}(t,x)|^2}{4 \EE |Z_{\varepsilon,\delta}(t,x)|^2}\,.
    \end{align*} 

    We  use the Chebyshev inequality to estimate $\tilde\xi_{n, \la}(t,x)$. We set
    \begin{equation*}
        \tilde I_{\varepsilon,\delta}(t,x)=\EE^{W'}\EE^{B^1,\cdots, B^4}\Bigg[ \sum_{j=1}^2\sum_{k=3}^4 Q^{(2)}(t,B^j,B^k)  \Theta(B^1,B^2, W_{\vare, \de}' ) \Theta(B^3,B^4, W_{\vare, \de}' )  \Bigg]\,.
    \end{equation*}
    By exchanging the order of integrations, we obtain
     \begin{align*}
     \tilde I_{\varepsilon,\delta}(t,x)
     &=\EE^{B^1,\cdots, B^4}  \lt[ \prod_{j=1}^4 u_0(B^j_t)  Q^{(2)}_{\varepsilon,\delta}(t, B^1,B^2)Q^{(2)}_{\varepsilon,\delta}(t, B^3,B^4) \sum_{j=1}^2\sum_{k=3}^4 Q^{(2)}(t,B^j,B^k) e^{Q^{(4)}_{\varepsilon,\delta}(t,B^1,\dots,B^4)}  \rt] 
     \\ &=4\EE^{B^1,\cdots, B^4}  \lt[ \prod_{j=1}^4 u_0(B^j_t)  Q^{(2)}_{\varepsilon,\delta}(t, B^1,B^2)Q^{(2)}_{\varepsilon,\delta}(t, B^3,B^4)  Q^{(2)}(t,B^1,B^3) e^{Q^{(4)}_{\varepsilon,\delta}(t,B^1,\dots,B^4)}  \rt]\,.
    \end{align*}
    Using \eqref{e.uniform-Q-bound}, we get $\tilde I_{\varepsilon,\delta}(t,x)\le \tilde I(t,x)$, where $\tilde I$ is defined in \eqref{def.tiI}.
    Thus, by Chebyshev inequality, we have
    \[
      \tilde\xi_{n, \la}(t,x) \le\frac{4\tilde I(t,x)}{\lambda|\EE Z_{\varepsilon,\delta}(t,x)|^2}\,.
    \]

    Combining previous estimates, we have 
    \begin{equation}\label{tmp:tiA}
    P(\tilde A_\la(t,x)) \ge \frac14 \frac{|\EE Z_{\varepsilon,\delta}(t,x)|^2}{ \EE |Z_{\varepsilon,\delta}(t,x)|^2}-\frac 4 \lambda\frac{\tilde I(t,x)}{ |\EE Z_{\varepsilon,\delta}(t,x) |^2 }\,. 
    \end{equation}
    Hence, if $\lambda$ satisfies \eqref{eqn:ltia}, then
    \begin{equation*}
     \frac 4 \lambda\frac{\tilde I(t,x) }{ |\EE Z_{\varepsilon,\delta}(t,x) |^2 }\le \frac18 \frac{|\EE Z_{\varepsilon,\delta}(t,x)|^2}{ \EE |Z_{\varepsilon,\delta}(t,x)|^2}
    \end{equation*}
    and the estimate \eqref{tmp:tiA} implies \eqref{est:tiA}.
  \end{proof}
 The above lemma and the concentration inequality yield 
  \begin{lemma}\label{lem:dtia} For every $a>0$ and $\lambda\ge \lambda_{\varepsilon,\delta}(t,x) $, we have
    \begin{equation}
      P \left( d(\WW,\tilde A_\lambda(t,x))>a+2\sqrt{\log\frac2{b_{\varepsilon,\delta}(t,x)}} \right)\le 2e^{-\frac{a^2}4}\,. 
    \end{equation}
  \end{lemma}
  \begin{theorem}\label{thm:Zsmall} Assume that Hypotheses \ref{h2}, \ref{hspa} and \ref{hroot} hold. Let $t>0$ and $x\in\RR^d$ be fixed. For every $0<r<\frac12 e^{-2\sqrt{\lambda(t,x)\log\frac2{b(t,x)}}}$, we have
    \begin{equation}\label{eqn:DUleftr}
      P \left( \frac{\|Du(t,x)\|_{\cH}^2}{\EE \|Du(t,x)\|_{\cH}^2} < r \right)
      \le 2\exp\lt\{-\frac14\lt(\frac{\log(2r)}{\sqrt{\lambda(t,x)}}+2\sqrt{\log\frac2{b(t,x)}} \rt)^2 \rt\} \,.
    \end{equation}
    Written another way, for every $a>0$, we have
    \begin{equation}
      P \left( \frac{\|Du(t,x)\|_{\cH}^2}{\EE \|Du(t,x)\|_{\cH}^2}<\frac12 e^{-2\sqrt{\lambda(t,x)\log\frac2{b(t,x)}}} e^{-\sqrt{\lambda(t,x)}a} \right)\le 2e^{-\frac{a^2}4}  \,.
    \end{equation}
  \end{theorem}
  \begin{proof}
  	We recall that $Z(t,x)=\|D^\WW u(t,x)\|_{L^2(\RR^{d+1})}^2=\|Du(t,x)\|_{\cH}^2$.
    Let $\varepsilon,\delta>0$ be fixed and $r$ be such that $0<r<\frac12 e^{-2\sqrt{\lambda_{\varepsilon,\delta}(t,x)\log\frac2{b_{\varepsilon,\delta}(t,x)}}}$.
    Applying Lemma \ref{lem:Duphi}, we have
    \begin{align*}
      P\lt(\frac{Z_{\varepsilon,\delta}(t,x)}{\EE Z_{\varepsilon,\delta}(t,x)}\le r\rt)
      \le P\lt(d(\WW_{\vare, \de}, \tilde A_{\la_{\varepsilon,\delta}(t,x)}(t,x))\ge \frac{-\log(2r)}{\sqrt {\la(t)}}\rt)\,. 
    \end{align*}
    Due to the range of $r$, we can choose $a=a(r)>0$ such that $a+2\sqrt{\log\frac2{b_{\varepsilon,\delta} (t,x)}}=-\frac{\log(2r)}{\sqrt{\lambda_{\varepsilon,\delta} (t,x)}} $ and apply Lemma \ref{lem:dtia} to obtain 
    \begin{equation*}
      P\lt(\frac{Z_{\varepsilon,\delta}(t,x)}{\EE Z_{\varepsilon,\delta}(t,x)}\le r\rt)
       \le 2 e^{-\frac{a^2}4}\,.
    \end{equation*} By sending $\varepsilon,\delta$ to $0$ and applying Proposition \ref{p.5.1} we obtain \eqref{eqn:DUleftr}.
  \end{proof}

  In the case when $u_0$ is bounded above and below by positive constants, it is possible to obtain more explicit estimates for $\lambda(t,x)$ and $b(t,x)$. For this purpose, we need the following lemma.
  \begin{lemma}\label{lem:Dubelow} Suppose that Hypotheses \ref{h2}, \ref{hspa} and \ref{hroot} hold and $u_0$ is bounded  below by a positive constant. Then, there exists a positive constant $C$ such that
    \begin{equation}
      \EE\|Du(t,x)\|^2_\cH\ge Ct^{2- \alpha_0-\frac \alpha2}
    \end{equation}
    for every $(t,x)\in\RR_+\times\RR^d$.
  \end{lemma}
  \begin{proof}
    From the identity \eqref{eqn:FKDuep}, we obtain the following Feynman-Kac presentation of $\|Du (t,x)\|^2_\cH$
    \begin{equation}
      \EE\|Du(t,x)\|^2_\cH=\EE\lt[u_0(B^1_t)u_0(B^2_t)Q^{(2)}(t,B^1,B^2)e^{Q^{(2)}(t,B^1,B^2)} \rt]\,,
    \end{equation}
    where $B^1,B^2$ are independent Brownian motions in $\RR^d$ starting at $x$. Since $Q^{(2)} $ is non-negative and $u_0$ is bounded below, we see that
    \begin{equation*}
       \EE\|Du(t,x)\|^2_\cH\ge C\EE Q^{(2)}(t,B^1,B^2)\,.
    \end{equation*}
    We consider the set $A_{\varepsilon,t}=\{\sup_{0\le s,r\le t}|B^1_s-B^2_r|\le \varepsilon \}$.
    As in \cite{HHNT}*{page 46}, we see that $P(A_{\varepsilon,t})\ge Ce^{-\frac {Ct}{\varepsilon^2}} $.
    Owning to Hypothesis \ref{h2}, on $A_{\varepsilon,t}$ we have
    \[Q^{(2)}(t,B^1,B^2)\ge c_0\int_0^t\int_0^t \gamma_0(s-r)\varepsilon^{-\alpha}dsdr\ge C \varepsilon^{-\alpha}t^{2- \alpha_0} .\]
    Then the above estimates imply that
    \begin{align*}
      \EE\|Du(t,x)\|^2_\cH
      \ge C t^{2- \alpha_0}\varepsilon^{-\alpha}\exp\lt\{-\frac{Ct}{\varepsilon^2} \rt\}\,.
    \end{align*}
    By choosing $\varepsilon=t^{\frac12}$, we see that the above estimate implies the result.
  \end{proof}
  Recall that $\lambda(t,x)$ and $b(t,x)$ are defined in \eqref{def:lb}.
  \begin{proposition}\label{prop.lb} Suppose that Hypotheses \ref{h2}, \ref{hspa} and \ref{hroot}  hold and that $u_0$ is bounded above and below by positive constants. Then, there exist positive constants $C_1,C_2,c_1,c_2$ such that 
    \begin{equation}\label{eqn.lb}
      \lambda(t,x)\le C_1t^{2- \alpha_0-\frac \alpha2} e^{c_1 t^\beta}
      \quad\textrm{and}\quad b(t,x)\ge C_2e^{-c_2t^\beta}
      \quad\textrm{for all}\quad(t,x)\in\RR_+\times\RR^d\,.
    \end{equation}
  \end{proposition}
  \begin{proof}
    First, we estimate $\tilde I$ from above.
    Applying Cauchy-Schwarz inequality, we have
    \begin{align*}
      \tilde I(t,x)\le 4 \|u_0\|_{\infty}^4\varepsilon\EE \lt[\lt(Q^{(2)}(t,B^1,B^2)Q^{(2)}(t,B^3,B^4)Q^{(2)}(t,B^1,B^3)\rt)^2\rt] +4\|u_0\|_{\infty}^4\varepsilon^{-1}\EE e^{2Q^{(4)}(t,B^1,\cdots,B^4)}.
    \end{align*}
    Applying Hypothesis \ref{h2} and Theorem \ref{t.hhnt} (ii), we see that $\tilde I(t,x)\le C  \varepsilon t^{6(2- \alpha_0-\frac \alpha2)}+ C \varepsilon^{-1}e^{2c t^\beta}$
    for some constants $C,c>0$. We then choose $\varepsilon=C t^{-3(2- \alpha_0-\frac \alpha2)}e^{ct^\beta}$ to obtain
    \begin{align*}
      \tilde I(t,x)\le C t^{3(2- \alpha_0-\frac \alpha2)}e^{c t^\beta}\,.
    \end{align*}
    The former estimate in \eqref{eqn.lb} follows by applying Corollary \ref{cor.Dku-bound}, Lemma \ref{lem:Dubelow} and the previous estimate. The later estimate in \eqref{eqn.lb} follows from Corollary \ref{cor.Dku-bound} and Lemma \ref{lem:Dubelow}.
  \end{proof}
  As an immediate application of Proposition \ref{prop.lb} and Theorem \ref{thm:Duneg}, we have
  \begin{corollary}\label{cor.Du-k}
   Suppose that Hypotheses \ref{h2}, \ref{hspa} and \ref{hroot}  hold and that $u_0$ is bounded above and below by positive constants. For every $T>0$ and $p>0$, there exist positive constants $C,c_T$ such that 
   \begin{equation}
     \EE\|Du(t,x)\|^{-2p}_\cH\le C e^{c_T p^2}t^{-p(2- \alpha_0-\frac \alpha2)} \quad\textrm{for all }t,x\in[0,T]\times\RR^d\,.
   \end{equation}
  \end{corollary}

\section{Tails  of the density}   \label{sec:tails}  
In this section, we   use Malliavin calculus to obtain estimates on the density of the solution to \eqref{e.1.1}, which corresponds to the  case $\its=1$.
Let $\rho(t,x; y)$ denote the density of the random variable $u(t,x)$,
which is the solution to \eqref{e.1.1}. Hypotheses \ref{h2}, \ref{hspa} and \ref{hroot} (as always) are assumed. 
\subsection{Right tail}
\begin{theorem}\label{t.7.2}   Let the initial condition $u_0$ be bounded from above
and from below by positive constants. 
Then, the law of the random variable $u(t,x)$ has a density $\rho(t,x; y)$ with respect to the Lebesgue measure,
namely, for any Borel set $A\subset \RR $, $P(u(t,x)\in A)=\int_A \rho(t,x; y) dy$.  Moreover, for every $T>0$,  there are  positive constants $c_1(T), c_2, c_3 >0$ and $\tilde c_1(T), \tilde c_2, \tilde c_3 >0$  such that 
  \begin{equation}\label{est.dsty.upper}
    \rho(t,x; y)  \le c_1(T) t^{-\frac{4-2 \alpha_0- \alpha}4} \exp\left\{-c_2 t^{-\frac{4-2 \alpha_0- \alpha}2}(\log (c_3y))^{\frac{4-\alpha} {2}}\right\} \quad\mbox{for every } t\in(0,T]\mbox{ and }y>a_0e^{b_0t^\beta}
  \end{equation}
  and
  \begin{equation}\label{est.dsty.lower}
    \rho(t,x;y)\ge \tilde c_1(T) t^{-\frac{4-2 \alpha_0- \alpha}2} \exp\left\{-\tilde c_2 t^{-\frac{4-2 \alpha_0- \alpha}2}(\log (\tilde c_3y))^{\frac{4-\alpha} {2}}\right\} 
  \end{equation} 
  for every $t\in(0,T]$ and $y>a_0e^{b_0t^\beta}\vee(p_t*u_0(x)+1)$.
\end{theorem}
\begin{proof}
First, we prove the upper bound.  We use the following identity (see \cites{hubook,nuabook}) 
\begin{eqnarray}
 \rho(t,x; y)&=&   \EE\left\{ I_{\{u(t,x)\ge 
 y\}}\delta\left(\frac{Du(t,x)}{\|Du(t,x)\|_\cH^2}\right)\right\} \nonumber \\
 &=&     \EE\left\{ I_{\{u(t,x)\ge
 y\}}\left[ \|Du(t,x)\|_\cH^{-2}   \delta\left( Du(t,x) \right) 
\right.\right. \nonumber\\
  & & \left.\left.  - \langle D\left( \|Du(t,x)\|_\cH^{-2} \right)\,,
 Du(t,x)\rangle _\cH\right] \right\}   =    \EE \bigg\{I_{\{u(t,x)\ge
 y\}} \cA\bigg\}
 \,,
 \label{e.rho-formula-add} 
\end{eqnarray}
where
\begin{equation}
\cA:= \|Du(t,x)\|_\cH^{-2}  \delta(D u(t,x))  +2 \|Du(t,x)\|_\cH^{-3} \langle   D^2u(t,x)   \,,
 Du(t,x)\otimes Du(t,x)\rangle _{\cH^{\otimes 2}} \,.
 \label{e.cA-definition}
\end{equation}
Applying Meyer's inequality and Corollary \ref{cor.Dku-bound}, we have
\begin{align*}
  \EE|\delta(Du(t,x))|^p
  \lesssim_{p} \EE\|Du(t,x)\|_{\cH}^p+\EE\|D^2u(t,x)\|_{\cH^{\otimes2}}^p
  \lesssim_{p,T} t^{\frac{4-2 \alpha_0- \alpha}4p}
\end{align*}
for all $p>1$, $t\in[0,T]$ and $x\in\RR^d$.
From   H\"older's  inequality, the previous estimate, Corollary \ref{cor.Dku-bound}, 
and Corollary \ref{cor.Du-k}   we have 
\begin{equation}
\|\cA\|_2 \lesssim_T t^{-\frac{4-2 \alpha_0- \alpha}4}\,. 
\label{e.cA-bound}
\end{equation}
for all $(t,x)\in(0,T]\times \RR^d$.
Thus,  by H\"older inequality and \eqref{e.lower-upper-tail} we have 
\begin{align*}
 \rho(t,x;y)
 &\le P(u(t,x)>y)^{1/2} \|\cA\|_2\\
 &\lesssim_T  t^{-\frac{4-2 \alpha_0- \alpha}4}   \exp\left(-c_2 t^{-\frac{4-2 \alpha_0- \alpha}{2}}(\log (c_3y))^{\frac{4- \alpha} {2}}\right)
\end{align*}
 which implies the upper bound \eqref{est.dsty.upper} (with a possibly different choice of $c$).

%
 	We now prove the lower bound \eqref{est.dsty.lower}. 
	Fix $(t,x)$ and   put
	\begin{equation}\label{white.varphi}
		\varphi(\tau,\xi,t,x)=p_{t- \tau}(x- \xi)u(\tau,\xi) \,.
	\end{equation}
	It is easily seen that $\varphi(\cdot,t,x)$ belongs to $\DD^{1,2}$. In addition,
from the following equation 
\begin{equation}
    \label{eqn.u.white}
    u(t,x)=p_t* u_0(x)+\int_0^t\int_{\RR^d} p_{t-s}(x-y)u(s,y) W(dy,ds) 
\end{equation}
we deduce that
	\begin{equation*}
		\int_0^t\int_\RR \varphi(\tau,\xi,s,y)W(d\xi,d\tau)=u(t,x)-p_t* u_0(x)\,.
	\end{equation*}
For any  $\HH$-valued random variable $F$ in $\DD^{1,2}$,   it is well-known that (see \cites{hubook,nuabook})
	\begin{align*}
		\rho(t,x; y)&=\EE \left[I_{\{u(t,x)\ge y\}}\delta\left(\frac{F}{\langle Du(t,x),F\rangle_\HH} \right)\right]
		\\&= \EE\left[I_{\{u(t,x)\ge y\}}\left(\frac{\delta(F)}{\langle Du(t,x),F\rangle_\HH}+\frac{\langle D\langle Du(t,x),F\rangle_\HH,F\rangle_\HH}{\langle Du(t,x),F\rangle_\HH^2}  \right)\right]
	\end{align*}
	With the choice 
\[
F(\tau,\xi) =\varphi(\tau,\xi,t,x)=1_{[0,t]}(\tau)p_{t- \tau}(x- \xi)u(\tau,\xi)
\]
 this yields
	\begin{equation}\label{white.rho}
		\rho(t,x; y)=\EE\left[I_{\{u(t,x)\ge y\}}\left(\frac{u(t,x)-p_t* u_0(x)}{A(t,x)}+\frac{\langle D_\cdot A(t,x),\varphi(\cdot,t,x)\rangle_\HH}{A(t,x)^2} \right) \right]\,,
	\end{equation}
	where 
	\begin{equation}
 A(t,x)=\langle D   u(t,x),\varphi(\cdot,t,x)\rangle_\HH\,. 
	\end{equation}  
It is easy to see that $A(t,x)\in L^p$ for any $p\in [1, \infty)$
and $ A(t,x)\ge 0$.  

	From  \eqref{white.varphi}, it follows that
	\begin{equation}\label{white.Atx}
		A(t,x)=\int_{\RR^2} \int_{\RR^{2d}}  D_{\tau, \xi} u(t,x) 
		1_{[0,t]}(\tau') p_{t- \tau' }(x- \xi')u(\tau' ,\xi') \gamma_0(\tau- \tau') \gamma(\xi-\xi')d \xi d \xi'd \tau d \tau'  \,.
	\end{equation}
Its Malliavin derivative is 
\begin{align*} 
		D_{\la, \eta} A(t,x)
  &=\int_{\RR^2} \int_{\RR^{2d}}   \Bigg[(D_{\tau, \xi, \la, \eta}^2 u(t,x)1_{[0,t]}^{\tau'} p_{t- \tau' }(x- \xi')u(\tau' ,\xi')  
  \gamma_0(\tau- \tau') \gamma(\xi-\xi')d \xi d \xi'd \tau d \tau' \\
  &\quad +\int_{\RR^2} \int_{\RR^{2d}}  D_{\tau, \xi} u(t,x)  
		1_{[0,t]}(\tau') p_{t- \tau' }(x- \xi') D_{  \la, \eta} u(\tau ,\xi')  
    \gamma_0(\tau- \tau') \gamma(\xi-\xi')d \xi d \xi' d \tau d \tau' \,.
	\end{align*}
It is straightforward to verify that  
\begin{equation*} 
	D_{\la, \eta}u(t,x) \ge 0  
\quad{\rm and}\quad 	
	D_{\tau, \xi, \la, \eta}^2 u(t,x)\ge 0\,.
\end{equation*} 
This implies that  $D_{\la, \eta} A(t,x) \ge 0$. Clearly, we have $F\ge 0$ and hence
\begin{equation}
\langle D_\cdot A(t,x),\varphi(\cdot,t,x)\rangle_\HH\ge 0
\,.
\label{e.7.11-add}
\end{equation} 
As a consequence, \eqref{white.rho} implies
	\begin{equation*} 
		\rho(t,x; y)\ge \EE \left[I_{\{u(t,x)\ge y\}}\frac{u(t,x)-p_t* u_0(x)}{A(t,x)}\right]\,.
	\end{equation*}
When $y>p_t*u_0(x)+1$, we have $u(t,x)-p_t*u_0(x)>1$ on the event $\{u(t,x)>y\}$. 
This means 
\begin{equation}\label{white.est.rho}
		\rho(t,x; y)\ge \EE \left[ \frac{I_{\{u(t,x)\ge y\}}}{A(t,x)}\right]\,.
	\end{equation}
Applying H\"older inequality,
	\begin{align*}
		\left(\EE \left[I_{\{u(t,x)\ge y\}} \right]\right)^2\le  \EE \left( \frac{I_{\{u(t,x)\ge y\}}}{A(t,x)}\right)\left(\EE A(t,x)\right)\,.
	\end{align*}
   	Thus
	\begin{equation}
		\rho(t,x; y)
    \ge \frac{\left(P(u(t,x)\ge y) \right)^2}{\EE A(t,x)} \,.\label{e.6.q-lower} 
	\end{equation} 

	
The numerator is bounded by Theorem \ref{t.3.12} (with $\beta=1$).
It remains to estimate $\EE A(t,x)$ from above. From \eqref{white.Atx}, we immediately  have 
\begin{align*} 
		\EE A(t,x)
		\le \lt\|\|Du(t,x)\|_\HH \rt\|_2\lt\|\|\varphi(\cdot,t,x)\|_\HH\rt\|_2 \,. 
\end{align*}
The above first factor is bounded by Corollary  \ref{cor.Dku-bound}.  
For the above  second factor, we have 
\begin{align*}
\|\|\varphi(\cdot,t,x)\|_\HH\|_2^2
&\le  c\int_0^t\int_0^t \int_{\RR^{2d}}  p_{t-\tau} (x-\xi) p_{t- \tau' }(x- \xi')    
\gamma_0(\tau- \tau') \gamma(\xi-\xi')d \xi d \xi' d \tau d \tau' 
 \\&= Ct^{\frac{4-2 \alpha_0- \alpha}2}\,. 
\end{align*} 
Summarizing the above estimates we see that 
\begin{equation*}
		\EE A(t,x) \le ct^{\frac{4-2 \alpha_0- \alpha}2 }   \,.
	\end{equation*} 
Together with \eqref{e.6.q-lower}  and the lower bound in
\eqref{e.lower-upper-tail},  this shows 
\[
\rho(t,x;y)\ge \tilde c_1 t^{-\frac{4-2 \alpha_0- \alpha}2  }\exp\left(-\tilde c_2  t^{-\frac{4-2 \alpha_0- \alpha}{2}}(\log (\tilde c_3 y))^{\frac{4-\upsilon} {2}}\right)\,,
\]
which is \eqref{est.dsty.lower}.
\end{proof}

\begin{remark} \label{rem:asymp}
(i) In the case of space-time white noise in spatial dimension one, Theorem \ref{t.7.2} yields
\begin{equation} 
	   C_1\exp\left(-c_1  t^{-1/2}(\log y)^{\frac{3} {2}}\right)\le 
 \rho(t,x; y)  \le     C_2   \exp\left(-c_2  t^{-1/2}(\log y)^{\frac{3} {2}}\right)  \label{est.density.white-add}
	\end{equation} 
	for sufficiently large $y$.

\noindent(ii) If $(t,x)$ is fixed, the two estimates \eqref{est.dsty.upper} and \eqref{est.dsty.lower} imply that
\begin{equation}
	-\tilde c_2 t^{-\frac{4-2 \alpha_0- \alpha}2} \le\liminf_{y\to\infty}\frac{\log \rho(t,x;y)}{(\log y)^{\frac{4- \alpha}2}}\le \limsup_{y\to\infty}\frac{\log \rho(t,x;y)}{(\log y)^{\frac{4- \alpha}2}}\le-c_2t^{-\frac{4-2 \alpha_0- \alpha}2} \,.
\end{equation}
(iii) If $(x,y)$ is fixed such that $y>a_0\vee(u_0(x)+1)$, the two estimates \eqref{est.dsty.upper} and \eqref{est.dsty.lower} imply that
\begin{equation}
	-\tilde c_2(\log(\tilde c_3 y))^{\frac{4- \alpha}2} \le\liminf_{t\to0} t^{\frac{4-2 \alpha_0- \alpha}2}\log  \rho(t,x;y) \le \limsup_{t\to0} t^{\frac{4-2 \alpha_0- \alpha}2}\log  \rho(t,x;y) \le-c_2(\log(c_3 y))^{\frac{4- \alpha}2}\,.
\end{equation}
\end{remark} 
%
 

\subsection{Left-tail}

\begin{theorem}\label{thm:lefttail}   
For fixed $T>0$, there are positive constants $a_0,b_0$ and $C(T),c_1(T),c_2(T)$ such that for every $t\in(0,T]$, $x\in\RR^d$ and $0<y<a_0e^{-b_0t^\beta}$
\begin{equation}
\rho(t,x;y) 
 \le    
 C(T)   t^{-\frac{4-2\al_0-\al}{4}}  \exp\left\{ - \left( - c_1(T) \log y  -c_2(T) \right)^2\right\}  \,.  
\label{e.upper-left-bound-general-last}
\end{equation} 
In particular, when  the noise $\dot W$  is    one dimensional space-time white, we have
\begin{equation}
\rho(t,x;y) 
 \le   C(T)    t ^{ -\frac14  }    \exp\left\{ - \left( - c_1  \log y  -c_2   \right)^2\right\}\,.  \label{e.6.last}
\end{equation}
\end{theorem} 
\begin{proof} We use a formula similar to \eqref{e.rho-formula-add}:  
\begin{eqnarray}
 \rho(t,x; y)=   - \EE \bigg\{I_{\{u(t,x)\le 
 y\}} \cA\bigg\}
 \,,  
 \label{e.rho-formula-add-2} 
\end{eqnarray}
where $\cA$ is defined by \eqref{e.cA-definition} and bounded by 
\eqref{e.cA-bound}. 
The H\"older inequality yields
\begin{eqnarray}
 \rho(t,x; y)\le     P\left(  u(t,x)\le 
 y\right)^{1/2}  \|\cA\|_2 \lesssim t^{-\frac{4-2\al_0-\al}{4}}
 P\left(  u(t,x)\le 
 y\right)^{1/2}   
 \,,   
\end{eqnarray}  
 The probability 
   $P( u(t,x)\le 
 y)$ is bounded by \eqref{e.upper-left-bound-general}.   This yields 
 \eqref{e.upper-left-bound-general-last}. The estimate \eqref{e.6.last} is proved in the same way. 
\end{proof}

\bibliography{biblio}  
\end{document}